\documentclass[a4paper,english]{article}
\usepackage{lmodern}
\usepackage[T1]{fontenc}
\usepackage[latin9]{inputenc}
\usepackage{units}
\usepackage{amsmath}
\usepackage{amsthm}
\usepackage{amssymb}
\usepackage{stmaryrd}
\usepackage[all]{xy}

\makeatletter

\pdfpageheight\paperheight
\pdfpagewidth\paperwidth

\newcommand{\noun}[1]{\textsc{#1}}
\providecommand{\tabularnewline}{\\}

\numberwithin{equation}{section}
\numberwithin{figure}{section}
\theoremstyle{plain}
\newtheorem{thm}{\protect\theoremname}
\theoremstyle{definition}
\newtheorem{defn}[thm]{\protect\definitionname}
\theoremstyle{remark}
\newtheorem{rem}[thm]{\protect\remarkname}
\theoremstyle{definition}
\newtheorem{example}[thm]{\protect\examplename}
\theoremstyle{definition}
\newtheorem*{defn*}{\protect\definitionname}
\theoremstyle{plain}
\newtheorem{prop}[thm]{\protect\propositionname}
\theoremstyle{plain}
\newtheorem{fact}[thm]{\protect\factname}
\theoremstyle{plain}
\newtheorem{lem}[thm]{\protect\lemmaname}
\theoremstyle{plain}
\newtheorem{cor}[thm]{\protect\corollaryname}
\theoremstyle{remark}
\newtheorem*{rem*}{\protect\remarkname}

\usepackage{multirow}

\usepackage{quiver}
\usepackage{adjustbox}
\usepackage{caption}
\DeclareCaptionLabelSeparator{dash}{~--~}

\makeatother

\usepackage{babel}
\providecommand{\corollaryname}{Corollary}
\providecommand{\definitionname}{Definition}
\providecommand{\examplename}{Example}
\providecommand{\factname}{Fact}
\providecommand{\lemmaname}{Lemma}
\providecommand{\propositionname}{Proposition}
\providecommand{\remarkname}{Remark}
\providecommand{\theoremname}{Theorem}

\begin{document}
\global\long\def\pplus{\overset{{\scriptscriptstyle \perp}}{\oplus}}%
\global\long\def\oton#1#2{#1_{1},\dots,#1_{#2}}%
\global\long\def\rk{\text{rk}}%
\global\long\def\tr{\text{tr}}%
\global\long\def\adj{\text{adj}}%
\global\long\def\im{\text{Im}}%
\global\long\def\nin{\notin}%
\global\long\def\id{\text{Id}}%
\global\long\def\makbil{\text{parallel}}%
\global\long\def\meridian{\text{meridian}}%
\global\long\def\sgn{\text{Sgn}}%
\global\long\def\fix{\text{Fix}}%
\global\long\def\hess{\text{Hess}}%
\global\long\def\ooton#1#2{#1_{0},\dots,#1_{#2}}%
\global\long\def\otherwise{\text{Otherwise}}%
\global\long\def\foton#1#2#3{#1\left(#2_{1}\right),\dots,#1\left(#2_{#3}\right)}%
\global\long\def\norm#1{\left\Vert #1\right\Vert }%
\global\long\def\nil{\text{Nill}}%
\global\long\def\fix{\text{Fix}}%
\global\long\def\spec{\text{Spec}}%
\global\long\def\ind{\text{Ind}}%
\global\long\def\fold{\text{Fold}}%
 
\global\long\def\hom{\text{Hom}}%
\global\long\def\ob{\text{Ob}}%
\global\long\def\coker{\text{coker}}%
\global\long\def\rad{\text{Rad}}%
\global\long\def\supp{\text{Supp}}%
\global\long\def\aut{\text{Aut}}%
\global\long\def\gal{\text{Gal}}%
\global\long\def\ann{\text{Ann}}%
\global\long\def\mayervi{\text{Mayer-Vietoris}}%
\global\long\def\conv{\text{conv}}%
\global\long\def\diam{\text{diam}}%
\global\long\def\length{\text{length}}%
\global\long\def\tp{\text{tp}}%
\global\long\def\lcm{\text{lcm}}%
\global\long\def\core{\text{Core}}%
\global\long\def\ad{\text{ad}}%
\global\long\def\ord{\text{ord}}%
\global\long\def\rank{\text{rank}}%
\global\long\def\uindex#1#2{\overset{{\scriptscriptstyle #2}}{#1}}%
\global\long\def\nuindex#1#2{\overset{{\scriptscriptstyle #2}}{#1}\,^{-1}}%
\global\long\def\involution{\overline{{\scriptstyle \bigbox}}}%
\global\long\def\mold{{\scriptstyle \varodot}}%
\global\long\def\type{\text{Type}}%
\global\long\def\spn{\text{span}}%

\title{A representation of $\text{Out}\left(F_{n}\right)$ by counting subwords
of cyclic words}
\author{Noam M.D. Kolodner}
\maketitle
\begin{abstract}
We generalize the combinatorial approaches of Rapaport and Higgins--Lyndon
to the Whitehead algorithm. We show that for every automorphism $\varphi$
of a free group $F$ and every word $u\in F$ there exists a finite
multiset of words $S_{u,\varphi}$ satisfying the following property:
For every cyclic word $w$, the number of times $u$ appears as a
subword of $\varphi\left(w\right)$ depends only on the appearances
of words in $S_{u,\varphi}$ as subwords of $w$. We use this fact
to construct a faithful representation of $\text{Out}\left(F_{n}\right)$
on an inverse limit of $\mathbb{Z}$-modules, so that each automorphism
is represented by sequence of finite rectangular matrices, which can
be seen as successively better approximations of the automorphism. 
\end{abstract}

\section{Introduction}

In \cite{MR1503309} Whitehead introduced an algorithm for determining
whether a word in the free group is primitive. Whiteheads proof of
the algorithm used 3-manifolds. Rapaport \cite{MR131452} Higgins
and Lyndon \cite{MR340420} came up with a combinatorial approach
for proving the algorithm this is summarized in Lyndon-Schupp \cite{MR1812024}.
In this paper we present generalization of this combinatorial approach.
 We will treat words as folded labeled graphs homeomorphic to a compact
line segment and cyclic words as folded labeled graphs homeomorphic
to a circle. One can interpret the method presented in Lyndon and
Schupp in the following way. Let $w$ be a cyclic word and let $\tau$
be a Whitehead generator. One can count how many times a specific
letter appears in $\tau\left(w\right)$ by counting how many times
appropriate 2-letter words appear in $w$. Furthermore, we define
$M_{k}$ with $k=1,2$ to be a $\mathbb{Z}-$module generated freely
by words of length $k$. For every cyclic word $w$ we can assign
an integer non negative vector $\pi_{k}\left(w\right)$ in $M_{k}$
which counts how many times each subword of length $k$ appears in
$w$. In this interpretation a whitehead generator $\tau$ gives rise
to a linear function $\tau_{1}:M_{2}\to M_{1}$ such that $\tau_{1}\pi_{2}\left(w\right)=\pi_{1}\left(\tau\left(w\right)\right)$. 

In Section 1 we show that for every automorphism $\varphi$ and any
word $u$ one can count the number of time $u$ appears in $\varphi\left(w\right)$
by counting appropriate subwords in $w$. 
\begin{thm}
\label{thm:main_counting_words}Let $\varphi$ an automorphism and
$u$ a word we show that there is a unique minimal finite multiset
of words $S_{u,\varphi}$ such that for every cyclic word $w$ 
\begin{equation}
\left|\hom\left(u,\varphi\left(w\right)\right)\right|=\sum_{v\in S_{u,\varphi}}\left|\hom\left(v,w\right)\right|\label{eq:main}
\end{equation}
\end{thm}

Because we consider words and cyclic words as labeled graphs, by counting
the number of graph morphisms between $u$ and $w$ we count the number
of time $u$ appears as a subword in $w$ . Given $u$ and $\varphi$
we will also show how to algorithmically calculate $S_{u,\varphi}$.
In Section 2 we use Equation \ref{eq:main} to construct a representation
of $Out\left(F_{n}\right)$. Let $\mathbb{Z}\left[C\right]$ be a
$\mathbb{Z}$-module generated freely by all cyclic words. We extend
$\pi_{k}$ linearly to an homomorphism $\pi_{k}:\mathbb{Z}\left[C\right]\to M_{k}$
while extending the definition of $M_{k}$ for $2\leq k$. 
\begin{thm}
\label{thm:main_inverse_limit}There are homomorphisms $p_{k}:M_{k}\to M_{k-1}$
that satisfy 
\begin{equation}
p_{k}\circ\pi_{k}=\pi_{k-1}\label{eq:inverselimit}
\end{equation}
 
\end{thm}

Moreover we construct $p_{k}$ and prove properties of this construction
in Theorem \ref{thm:mainsection2}. Let $\varphi$ be an automorphism
let $m_{k}^{\varphi}=\max_{u\in W_{k}}\max_{v\in S_{u,\varphi}}l\left(v\right)$
with $W_{k}$ denoting the set of words of length $k$ and $l\left(v\right)$
denoting the length of $v$. We will show how to use Equation \ref{eq:main}
to define homomorphisms $\varphi_{k}:M_{m_{k}}\to M_{k}$ that satisfy
\begin{equation}
\varphi_{k}\circ\pi_{m_{k}}=\pi_{k}\circ\varphi\label{eq:representation}
\end{equation}
Denote $\im\pi_{k}=\tilde{M}_{k}$ and consider $\varprojlim\tilde{M}_{k}$
the inverse limit of $p_{k}:\tilde{M}_{k}\to\tilde{M}_{k-1}$ . Equations
\ref{eq:inverselimit} and \ref{eq:representation} show that $\varphi\mapsto\varprojlim\varphi_{k}$
is a representation $Out\left(F_{n}\right)\to\aut\left(\varprojlim\tilde{M}_{k}\right)$.
\begin{thm}
\label{thm:main_injective_representation}The representation $Out\left(F_{n}\right)\to\aut\left(\varprojlim\tilde{M}_{k}\right)$
is a faithful representation
\end{thm}

As a corollary the multisets $S_{u,\varphi}$ characterize $\varphi$
i.e. let $\varphi,\psi$ automorphisms in different conjugacy classes
then there exists a word $u$ such that $S_{u,\varphi}\neq S_{u,\psi}$.
Although we have not yet found application for this representation
we believe it holds valuable information on outer automorphisms.

\section{Preliminaries }
\begin{defn}[Graphs]
We use graphs in the sense of Serre \cite{MR0476875}: A \emph{graph}
$\Gamma$ is a set $V\left(\Gamma\right)$ of vertices and a set $E\left(\Gamma\right)$
of edges with a function $\iota\colon E\left(\Gamma\right)\to V\left(\Gamma\right)$
called the initial vertex map and an involution $\involution\colon E\left(\Gamma\right)\to E\left(\Gamma\right)$
with $\overline{e}\neq e$ and $\overline{\overline{e}}=e$. A \emph{graph
morphism} $f\colon\Gamma\to\Delta$ is a pair of set functions $f^{E}\colon E\left(\Gamma\right)\to E\left(\Delta\right)$
and $f^{V}\colon V\left(\Gamma\right)\to V\left(\Delta\right)$ that
commute with the structure functions. A \emph{path }in $\Gamma$ is
a finite sequence $\oton en\in E\left(\Gamma\right)$ with $\iota\left(\overline{e}_{k}\right)=\iota\left(e_{k+1}\right)$
for every $1\leq k<n$. The path is \emph{closed} if $\iota\left(\overline{e}_{n}\right)=\iota\left(e_{1}\right)$,
and \emph{reduced} if $e_{k+1}\neq\overline{e}_{k}$ for all $k$.
All the graphs in the paper are assumed to be connected unless specified
otherwise, namely, for every $e,f\in E\left(\Gamma\right)$ there
is a path $e,\oton en,f$.
\end{defn}

\begin{defn}[Labeled graphs]
Let $X$ be a set and let $X^{-1}$ be the set of its formal inverses.
We define $R_{X}$ to be the graph with $E\left(R_{X}\right)=X\cup X^{-1}$,
$V\left(R_{X}\right)=\left\{ *\right\} $, $\overline{x}=x^{-1}$
and $\iota\left(x\right)=*.$ An \emph{$X$-labeled graph }is a graph
$\Gamma$ together with a graph morphism $\Lambda\colon\Gamma\to R_{X}$
called the label function. A morphism of $X$-labeled graphs $\Gamma$
and $\Delta$ is a graph morphism $f\colon\Gamma\to\Delta$ that commutes
with the label functions. 
\end{defn}

\begin{defn}[Folding]
A labeled graph $\Gamma$ is \emph{folded }if $\Lambda\left(e\right)\neq\Lambda\left(f\right)$
holds for every two edges $e,f\in E\left(\Gamma\right)$ with $\iota\left(f\right)=\iota\left(e\right)$.
Or equivalently if $\Lambda$ is a graph immersion. If $\Gamma$ is
not folded, there exist $e,f\in E\left(\Gamma\right)$ s.t.\ $\iota\left(e\right)=\iota\left(f\right)$
and $\Lambda\left(e\right)=\Lambda\left(f\right)$; Let $\Gamma'$
be the graph obtained by identifying the vertex $\iota\left(\overline{e}\right)$
with $\iota\left(\overline{f}\right)$, the edges $e$ with $f$ and
$\overline{e}$ with $\overline{f}$. We say $\Gamma'$ is the result
of \emph{folding} $e$ and $f$. The label function $\Lambda$ factors
through $\Gamma'$, yielding a label function $\Lambda'$ on $\Gamma'$.
The labeled graphs morphism $g:\Gamma\to\Gamma'$ is called a folding.
Let $\Gamma$ be a finite labeled graph. There is a finite sequence
of foldings $g_{i}:\Gamma_{i-1}\to\Gamma_{i}$ with $\Gamma_{0}=\Gamma$,
$0\leq i\leq n$ such that $\Gamma_{n}$ is folded. The folded graph
$\Gamma_{n}$ is unique and it does not depend on the order in which
we choose the foldings. See \cite{MR695906}.
\end{defn}

\begin{rem}
For reduced words $u,v$ in a free group, we write $u\cdot v$ to
indicate that there is no cancellation in their concatenation, namely
the first letter in $v$ is not the inverse of the last letter in
$u$. 
\end{rem}

\begin{defn}[Proper subword]
 Let $F_{n}$ be a free and let $u_{1},u_{2},v,w\in F_{n}$ s.t.
$u_{1}\cdot v\cdot u_{2}=w$. If the multiplication $u_{1}\cdot v$
and $v\cdot u_{2}$ are preformed in $F_{n}$ without cancellation
$u_{1},v$ and $u_{2}$ are called proper subwords of $w$. We call
$u_{1}$ a prefix and $u_{2}$ a suffix. 
\end{defn}

\begin{defn}[Unoriented-Word]
 A unoriented word is a folded labeled graph that is topologically
homeomorphic to a connected compact set in $\mathbb{R}$. Let $u\in F_{n}$
be a word we denote the unoriented word by $\overline{u}=\overline{u^{-1}}$
. And call $u$ and $u^{-1}$ orientations of $\overline{u}$. When
context allows we drop the bar. We abbreviate u-word. 
\end{defn}

\begin{defn}[Cyclic Word]
 A cyclic word is a folded labeled graph that is topologically homeomorphic
to a circle. We denote a cyclic word by a cyclically reduced representative
without a bar. This may differ from convention as often cyclic words
are defined to be the conjugation class of a word in $F_{n}$. Let
$w\in F_{n}$ under our definition $w$ and $w^{-1}$ denote the same
cyclic word. 
\end{defn}

\begin{defn}[Subword]
 Let $w$ be a u-word. A u-word $u$ is called subword of $w$ if
there is a graph morphism $u\to w$. Let $w$ be a cyclic word. An
u-word $\overline{u}$ will be called a subword of $w$ if there is
a labeled graph morphism from $u$ to $w$ (it need not be injective!).
Let $x\in F_{\left\{ x,y\right\} }$ be a one letter cyclic word and
let $\overline{x^{n}}$ be an u-word under this definition $\overline{x^{n}}$
is a subword of $x$ (as a cyclic word of course). Again this may
differ from convention.
\end{defn}

\begin{defn}[k-Affix]
 Let $k\in\mathbb{N}$ and let $\overline{w}$ be an u-word with
$l\left(w\right)>k$. A $k$-affix is an u-word $\overline{u}$ of
length $k$ s.t. $u$ is either a prefix or suffix of $w$ or $w^{-1}$.
We notice that if $u$ is prefix of $w$ than $u^{-1}$ is a suffix
of $w^{-1}$ thus each u-word $\overline{w}$ has exactly two $k$-affixes.
This can be seen as a graph function between u-words that is a local
homeomorphism except for exactly one point.
\end{defn}

\begin{defn}
Let $\varphi\colon F_{Y}\to F_{X}$ be a homomorphism such that $\forall y\in Y,\varphi\left(y\right)\neq1$
. Let $\Delta$ be a $Y$-labeled graph we construct an $X$-labeled
graph $\mathcal{F}_{\varphi}\left(\Delta\right)$ by replacing every
edge $e\in E\left(\Delta\right)$ by a reduced path labeled $\varphi\left(\Lambda\left(e\right)\right)$.
Let $\Delta$ and $\Xi$ be $Y$-labeled graphs, and $f\colon\Delta\to\Xi$
a graph morphism. The graph morphism $\mathcal{F}_{\varphi}\left(f\right)\colon\mathcal{F}_{\varphi}\left(\Delta\right)\to\mathcal{F}_{\varphi}\left(\Xi\right)$
is defined by gluing the identity maps $\bigsqcup_{e\in E\left(\Delta\right)}\varphi\left(\Lambda\left(e\right)\right)\to\varphi\left(\Lambda\left(f\left(e\right)\right)\right)$.
We notice that generally $\mathcal{F}_{\varphi}\left(\Delta\right)$
is not a folded graph. We denote by $\tilde{\mathcal{F}}_{\varphi}\left(\Delta\right)$
the folded graph obtained by folding $\mathcal{F}_{\varphi}\left(\Delta\right)$.
It is a functor as well. We notice that the folding morphism denoted
by $\rho_{\Delta}:\mathcal{F}_{\varphi}\left(\Delta\right)\to\tilde{\mathcal{F}}_{\varphi}\left(\Delta\right)$
is a natural transformation. This definition is taken from \cite{MR4187248}.
\end{defn}

The motivation for $\mathcal{F}$ is topological: thinking of $\Delta,R_{Y},R_{X}$
as topological spaces and of $\Lambda\colon\Delta\to R_{Y},\:\varphi:R_{Y}\to R_{X}$
as continuous functions, we would like to think of $\Delta$ as an
$X$-labeled graph with the label function $\varphi\circ\Lambda$.
The problem is that $\varphi\circ\Lambda$ does not send edges to
edges, and we mend this by splitting edges in $\Delta$ to paths representing
their images in $R_{X}$.
\begin{example}
\label{exa:functor}Let $\varphi:F_{\left\{ x,y\right\} }\to F_{\left\{ x,y\right\} }$
be the automorphism $x\mapsto xy,\:y\mapsto y$. In Figure \ref{fig:functor_exmpl}
we see the graph morphism between the u-word $\overline{x}$ and the
cyclic word $xy$ and the morphism after applying\emph{ }$\mathcal{F}_{\varphi}$
and $\tilde{\mathcal{F}_{\varphi}}$. 
\end{example}

\captionsetup[figure]
{%
labelsep  = space  
}

\begin{figure}[h]
\adjustbox{scale=0.75,center}{%
\begin{tikzcd}
&&& \bullet & \bullet \\ 
&&& \bullet & \bullet \\ 
&&& {} & {} \\ 	&& {} &&& {} \\ 
\bullet & \bullet & \bullet &&& \bullet & \bullet & \bullet \\
& \bullet \\ 
\bullet && \bullet &&& {\scriptstyle x} & \bullet & \bullet
\arrow["x"{description}, from=1-4, to=1-5] 	
\arrow["x"{description}, curve={height=-6pt}, from=2-4, to=2-5] 	
\arrow["y"{description}, curve={height=12pt}, from=2-4, to=2-5] 	
\arrow["{\tilde{\mathcal{F}}_{\varphi}}", Rightarrow, from=3-5, to=4-6] 	
\arrow["x"{description}, from=5-6, to=5-7] 	\arrow["y"{description}, from=5-7, to=5-8] 	
\arrow[curve={height=-18pt}, from=7-6, to=7-7] 	\arrow[curve={height=-18pt}, no head, from=7-7, to=7-6] 	
\arrow["y"{description}, from=7-7, to=7-8] 	
\arrow["{\mathcal{F} _{\varphi}}"', Rightarrow, from=3-4, to=4-3] 	
\arrow["x"{description}, from=5-1, to=5-2] 	
\arrow["y"{description}, from=5-2, to=5-3] 	
\arrow["x"{description}, curve={height=-6pt}, from=7-1, to=6-2] 	
\arrow["y"{description}, curve={height=-6pt}, from=6-2, to=7-3] 	
\arrow["y"{description}, curve={height=12pt}, from=7-1, to=7-3] 
\end{tikzcd}
}
\caption{\label{fig:functor_exmpl}}
\end{figure}

\begin{defn}
Let $W_{k}$ be the set of length $k$ u-words. We define $M_{k}=\spn_{\mathbb{Z}}W_{k}$
the free module generated by $W_{k}$. We define $M_{k}^{*}$ to be
the dual to $M_{k}$. If $u\in W_{k}$ let $u^{*}\in M_{k}^{*}$ be
the functional s.t. $u^{*}\left(u\right)=1$ and $u^{*}\left(v\right)=0$
for every $v\in W_{k}$ such that $v\neq u$. Let $C$ the set of
cyclic words of $F_{n}$ and $\mathbb{Z}\left[C\right]$ the free
model generated by cyclic words. We define a sequence of homomorphisms
$\pi_{k}:\mathbb{Z}\left[C\right]\to M_{k}$. Let $w\in C$ and $u\in W_{k}$
we define $u^{*}\left(\pi_{k}\left(w\right)\right)=\left|\hom\left(u,w\right)\right|$ 
\end{defn}

\begin{example}
\label{exa:pi_k}Take $F_{\left\{ x,y\right\} }$and the cyclic word
$xyx^{-1}y^{-1}$ then $\pi_{2}\left(xyx^{-1}y^{-1}\right)=\overline{xy}+\overline{yx^{-1}}+\overline{x^{-1}y^{-1}}+\overline{y^{-1}x}$
and $\pi_{3}\left(xyx^{-1}y^{-1}\right)=\overline{xyx^{-1}}+\overline{yx^{-1}y^{-1}}+\overline{x^{-1}y^{-1}x}+\overline{y^{-1}xy}$
and $\pi_{1}\left(xyx^{-1}y^{-1}\right)=2\overline{x}+2\overline{y}$. 
\end{example}

\section{Counting subwords}

We start by demonstrating our interpretation of Lyndon-Schupp for
the simplest kind of Whitehead generator. 
\begin{example}
\label{exa:main}Let $\varphi:F_{\left\{ x,y,z\right\} }\to F_{\left\{ x,y,z\right\} }$
be the automorphism $x\mapsto xy,y\mapsto y,z\mapsto z$. We construct
$\varphi_{1}:M_{2}\to M_{1}$ s.t. $\varphi_{1}\pi_{2}\left(w\right)=\pi_{1}\left(\varphi\left(w\right)\right)$ 
\end{example}

Let $w$ be a cyclic word. The letters $\overline{x}$ and $\overline{z}$
(u-words of length $1$) will appear the same amount of times in $\varphi\left(w\right)$
as they appear in $w$. Each time $x$ appears in $w$ it is followed
by $x,y,y^{-1},z$ or $z^{-1}$ therefor by counting the number of
times the subwords $\overline{xx},\overline{xy},\overline{xy^{-1}},\overline{xz},\overline{xz^{-1}}$
appear we count the number of times $\overline{x}$ appears. Now we
count the number of appearances of $\overline{y}$ in $\varphi\left(w\right)$.
Every $y$ in $w$ that is not followed by $x^{-1}$ gives a $y$
in $\varphi\left(w\right)$. Every $x$ in $w$ gives a $y$ in $\varphi\left(w\right)$
as long as it is not followed by $y^{-1}$. So if we count the number
of the subwords $\overline{yx},\overline{yy},\overline{yz},\overline{yz^{-1}}$
and $\overline{xx},\overline{xy},\overline{xz},\overline{xz^{-1}}$
in $w$ this will gives us the number of appearances of $y$ in $\varphi\left(w\right)$.
We use the identification of $\hom\left(M_{2},M_{1}\right)$ with
$M_{2}^{*}\otimes M_{1}$ to define $\varphi_{1}$ as follows
\begin{eqnarray*}
\varphi_{1} & = & \left(\overline{xx}^{*}+\overline{xy}^{*}+\overline{xy^{-1}}^{*}+\overline{xz}^{*}+\overline{xz^{-1}}^{*}\right)\otimes\overline{x}\\
 & + & \left(\overline{yx}^{*}+\overline{yy}^{*}+\overline{yz}^{*}+\overline{yz^{-1}}^{*}+\overline{xx}^{*}+\overline{xy}^{*}+\overline{xz}^{*}+\overline{xz^{-1}}^{*}\right)\otimes\overline{y}\\
 & + & \left(\overline{zx}^{*}+\overline{zx^{-1}}^{*}+\overline{zy}^{*}+\overline{zy^{-1}}^{*}+\overline{zz}^{*}\right)\otimes\overline{z}
\end{eqnarray*}
 We want to generalize this. We will generalize in a naive intuitive
approach then show why this approach does not work and give the right
but somewhat complicated and non-intuitive definitions.
\begin{defn*}[naive ideal-preimage]
 Let $\varphi$ be an automorphism and $\overline{u}$ be an u-word.
An u-word $\overline{v}$ is a $\varphi$-ideal preimage of $\overline{u}$
if for every u-word $\overline{w}$ with a graph morphism $\overline{v}\to\overline{w}$
there exists a graph morphism $\overline{u}\to\overline{\varphi\left(w\right)}$.
Equivalently, let $v$ be an orientation of $\overline{v}$. For every
words $w_{1},w_{2}$ with $w_{1}\cdot v\cdot w_{2}$ the u-word $\overline{u}$
is a subword of $\varphi\left(\overline{w_{1}vw_{2}}\right)$. 
\end{defn*}
We notice that if $\overline{v_{1}}$ is a $\varphi$-ideal preimage
of $\overline{u}$ and $\overline{v_{2}}$ is a u-word such that there
is a graph morphism $\overline{v_{1}}\to\overline{v_{2}}$ then $\overline{v_{2}}$
is a $\varphi$-ideal preimage of $\overline{u}$ as well.
\begin{defn*}[naive Minimal ideal-preimage]
 A $\varphi$-ideal preimage $\overline{v}$ of $\overline{u}$ is
said to be minimal if there is no $\overline{v'}\to\overline{v}$
such that $\overline{v'}$ is also a $\varphi$-ideal preimage of
$\overline{u}$
\end{defn*}
In Example \ref{exa:main} the minimal $\varphi$-ideal preimages
of $\overline{x}$ and $\overline{z}$ are $\overline{x}$ and $\overline{z}$
respectively and the minimal $\varphi$-ideal preimages of $\overline{y}$
are $\overline{yx},\allowbreak\overline{yy},\allowbreak\overline{yz},\allowbreak\overline{yz^{-1}},\allowbreak\overline{xx},\allowbreak\overline{xy},\allowbreak\overline{xz},\allowbreak\overline{xz^{-1}}$.
Thus imitating Example \ref{exa:main} we count the number of subwords
$\overline{u}$ in a cyclic word $\varphi\left(w\right)$ by counting
the number of minimal $\varphi$-ideal preimages in $w$. The next
two example show why a more sophisticated approach is needed
\begin{example}
\label{exa:problem1}Let $\varphi:F_{\left\{ x,y,z\right\} }\to F_{\left\{ x,y,z\right\} }$
be the automorphism $x\mapsto xyy,\,y\mapsto y,\,z\mapsto z$. We
notice that as before in Example \ref{exa:main} the u-word $\overline{xz}$
is a $\varphi$-ideal preimage of $\overline{y}$ but we observe that
$\varphi\left(\overline{xz}\right)=\overline{xyyz}$ and that no matter
what words $w_{1},w_{2}$ with $w_{1}\cdot xz\cdot w_{2}$ then $\varphi\left(\overline{w_{1}xzw_{2}}\right)$
will include 2-copies of $\overline{y}$. Thus the approach of counting
the number of $\overline{y}$ in $\varphi\left(w\right)$ by counting
minimal ideal preimages will not work here. 
\end{example}

\begin{example}
\label{exa:problem2}Let $\varphi:F_{\left\{ x,y,z\right\} }\to F_{\left\{ x,y,z\right\} }$
be the automorphism $x\mapsto xy,y\mapsto y,z\mapsto xyz$. Here we
give an example of how minimal ideal preimages can be nested in each
other. We will state some facts without proof. One can prove these
facts using the tools developed in the continuation of this paper.
The u-word $\overline{xz}$ is a $\varphi$-ideal preimage of $\overline{xy}$.
We notice that $\varphi\left(\overline{xz}\right)=\overline{xyxyz}$
and for the $\overline{xy}$ marked with an underline $\overline{xy\underline{xy}z}$
and for any words $w_{1},w_{2}$ with $w_{1}\cdot xz\cdot w_{2}$
the u-word $\varphi\left(\overline{w_{1}xzw_{2}}\right)$ will contain
a copy of the underlined $\overline{xy}$. As for the other $\overline{xy}$
underlined in $\overline{\underline{xy}xyz}$ this is not true if
we take $w_{1}=z^{-1}$ then there is cancellation 
\[
\varphi\left(\overline{z^{-1}xz}\right)=\overline{z^{-1}y^{-1}x^{-1}xyxyz}=\overline{z^{-1}xyz}
\]
But if we take $\overline{yxz}$ then $\varphi\left(\overline{yxz}\right)=\overline{yxyxyz}$.
Now for every words $w_{1},w_{2}$ with $w_{1}\cdot yxz\cdot w_{2}$
the u-word $\varphi\left(\overline{w_{1}yxzw_{2}}\right)$ does contains
a copy of the $\overline{xy}$ underlined in $\overline{y\underline{xy}xyz}$.
Under the naive definition of a minimal ideal preimage $\overline{yxz}$
is not an ideal preimage because it contains the minimal ideal preimage
$\overline{xz}$. But then we cannot count the number of times $\overline{xy}$
appears by counting minimal ideal preimages. Because if $w$ is a
cyclic word with a graph morphism $\overline{yxz}\to w$ then $\overline{yxz}$
will contribute two copies of $\overline{xy}$ to $\varphi\left(w\right)$
but we will only count one. 
\end{example}

Let $\overline{u}$ be a u-word and $\varphi$ an automorphism and
let $\overline{v}$ be a $\varphi$-ideal preimage of $\overline{u}$.
We conclude from Examples \ref{exa:problem1} and \ref{exa:problem2}
that in order to count appearances of $\overline{u}$ in $\varphi\left(w\right)$
using ideal preimages our definition need to take into account the
specific embedding of $\overline{u}$ in $\varphi\left(v\right)$. 
\begin{defn}[Ideal-preimage]
 Let $\varphi$ be an automorphism of a free group. Let $u$ be an
u-word. An ideal preimage of $u$ with respect to $\varphi$ is a
pair of a u-word $v$ and a graph morphism $u\to\varphi\left(v\right)$
s.t. for every u-word $w$ with $v\to w$ there is a graph morphism
$u\to\varphi\left(w\right)$ such that the diagram  
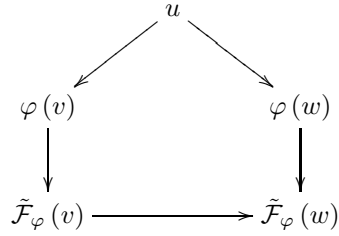
\[
\xymatrix{ & u\ar[rd]\ar[ld]\\
\varphi\left(v\right)\ar[d] &  & \varphi\left(w\right)\ar[d]\\
\mathcal{\tilde{F}}_{\varphi}\left(v\right)\ar[rr] &  & \mathcal{\tilde{F}}_{\varphi}\left(w\right)
}
\]
 commutes. Even though we defined an ideal preimage to be a pair $\left(v,u\to\varphi\left(v\right)\right)$
we will denote it $u\to\varphi\left(v\right)$ for brevity. We will
call the u-word $v$ the base word of $u\to\varphi\left(v\right)$.
\end{defn}

This definition compares a specific embedding of $u$ in $\varphi\left(v\right)$
with a specific embedding of $u$ in $\varphi\left(w\right)$ that
arises from the specific embedding of $v\to w$. In general a graph
morphism $v\to w$ does not give rise to a graph morphisms between
$\varphi\left(v\right)$ and $\varphi\left(w\right)$ but it does
give rise to a graph morphism $\mathcal{\tilde{F}}_{\varphi}\left(v\right)\to\mathcal{\tilde{F}}_{\varphi}\left(w\right)$
cf. Example \ref{exa:functor}. 

Two distinct ideal preimages can have the same base word $v$ but
two distinct graph morphisms as in Examples \ref{exa:problem1} and
\ref{exa:problem2}. An ideal preimage $u\to\varphi\left(v_{1}\right)$
is said to be a sub ideal preimage of $u\to\varphi\left(v_{2}\right)$
if there is a graph morphism $v_{1}\to v_{2}$ such that 
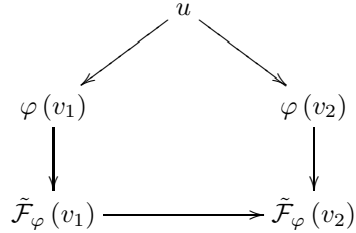
\[
\xymatrix{ & u\ar[rd]\ar[ld]\\
\varphi\left(v_{1}\right)\ar[d] &  & \varphi\left(v_{2}\right)\ar[d]\\
\tilde{\mathcal{F}}_{\varphi}\left(v_{1}\right)\ar[rr] &  & \mathcal{\tilde{F}}_{\varphi}\left(v_{2}\right)
}
\]
If $u\to\varphi\left(v_{1}\right)$ and $v_{2}$ is a u-word with
$v_{1}\to v_{2}$ then there is a morphism $u\to\varphi\left(v_{2}\right)$
satisfying the diagram above thus $u\to\varphi\left(v_{2}\right)$
inherits a structure of ideal preimage. If an ideal preimage has no
non-trivial sub ideal preimages we say that it is a minimal ideal
preimage. 
\begin{prop}
Let $u\to\varphi\left(v\right)$ be an ideal preimage then for every
\textbf{cyclic} word $w$ with $v\to w$ there is a morphism $u\to\varphi\left(w\right)$
such that 
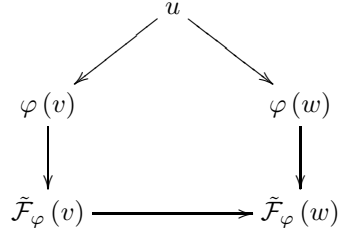
\[
\xymatrix{ & u\ar[rd]\ar[ld]\\
\varphi\left(v\right)\ar[d] &  & \varphi\left(w\right)\ar[d]\\
\tilde{\mathcal{F}}_{\varphi}\left(v\right)\ar[rr] &  & \mathcal{\tilde{F}}_{\varphi}\left(w\right)
}
\]
 
\end{prop}

\begin{proof}
Let $w'$ be a word that is a cyclically reduced representative of
$w$. The word $\varphi\left(w'\right)$ is not necessarily cyclically
reduced. We write $\varphi\left(w'\right)=t^{-1}\cdot\dot{\varphi}\left(w'\right)\cdot t$
with $\dot{\varphi}\left(w'\right)$ cyclically reduced. Let $k$
be large enough such that $v\to w'^{k}$ then by the definition of
ideal preimage there is a morphism $u\to\varphi\left(w'\right)^{k}$.
If we choose $k$ large enough we have a morphism $u\to\dot{\varphi}\left(w'\right)^{k}$
and there is a graph morphism from $\dot{\varphi}\left(w'\right)^{k}$
to the cyclic word $\varphi\left(w\right)$. It is easy to verify
that this graph morphism $u\to\varphi\left(w\right)$ satisfies the
desired diagram. 
\end{proof}
\begin{defn}[Full set of ideal preimages]
 Let $u$ be a u-word and $\varphi$ an automorphism. A set of ideal
preimages $S$  is called a full set if for every cyclic word $w$
with a morphism $u\to\varphi\left(w\right)$ there is a unique ideal
preimage $u\to\varphi\left(v\right)\in S$ such that $v\to w$ and
the diagram 
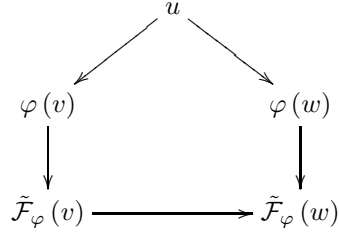
\[
\xymatrix{ & u\ar[rd]\ar[ld]\\
\varphi\left(v\right)\ar[d] &  & \varphi\left(w\right)\ar[d]\\
\tilde{\mathcal{F}}_{\varphi}\left(v\right)\ar[rr] &  & \mathcal{\tilde{F}}_{\varphi}\left(w\right)
}
\]
commutes.
\end{defn}

\begin{rem}
It is important that $w$ in the definition is a cyclic word. The
definition would be null if we change cyclic word by u-word. Let $\varphi:F_{\left\{ x,y\right\} }\to F_{\left\{ x,y\right\} }$be
the automorphism $x\mapsto xy,y\mapsto y$, $u$ be the u-word $u=\overline{y}$
and $w$ be the u-word $w=\overline{x}$. Then there is a graph morphism
$\overline{y}\to\varphi\left(\overline{x}\right)=\overline{xy}$ but
$\overline{x}$ is not an ideal preimage of $\overline{y}$. 
\end{rem}

\begin{defn}[A full paradigm]
 Let $\overline{u},\overline{v}$ be u-words, $\varphi$ be an automorphism,
$\overline{u}\to\varphi\left(\overline{v}\right)$ a graph morphism
and let $v$ be an orientation of $\overline{v}$. A full paradigm
of $\overline{u}\to\varphi\left(\overline{v}\right)$ is a set $\overline{u}\to\varphi\left(\overline{v_{1}}\right),\dots,\overline{u}\to\varphi\left(\overline{v_{2n-1}}\right)$
such that 
\[
\left\{ \overline{v_{1}},\dots,\overline{v_{2n-1}}\right\} =\left\{ \overline{vx}|x\in X\cup X^{-1},v\cdot x\right\} 
\]
 or 
\[
\left\{ \overline{v_{1}},\dots,\overline{v_{2n-1}}\right\} =\left\{ \overline{xv}|x\in X\cup X^{-1},x\cdot v\right\} 
\]
 and the diagram 
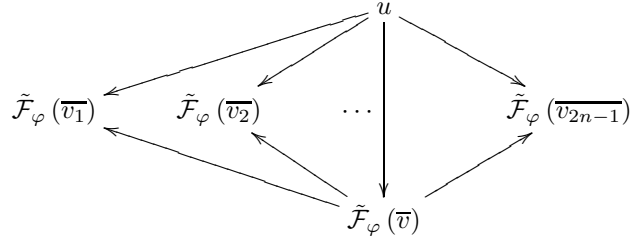
\[
\xymatrix{ &  & \overline{u}\ar[dr]\ar[dl]\ar[dll]\ar[dd]\\
\tilde{\mathcal{F}}_{\varphi}\left(\overline{v_{1}}\right) & \tilde{\mathcal{F}}_{\varphi}\left(\overline{v_{2}}\right) & \cdots\qquad & \mathcal{\tilde{F}}_{\varphi}\left(\overline{v_{2n-1}}\right)\\
 &  & \mathcal{\tilde{F}}_{\varphi}\left(\overline{v}\right)\ar[ur]\ar[ul]\ar[ull]
}
\]
commutes. 
\end{defn}

\begin{rem}
Let $u\to\varphi\left(v\right)$ and let $u\to\varphi\left(v_{1}\right),\dots,u\to\varphi\left(v_{l}\right)$
be a full paradigm. It is not hard to see that $u\to\varphi\left(v\right)$
is an ideal preimage if and only if $u\to\varphi\left(v_{1}\right),\dots,u\to\varphi\left(v_{l}\right)$
are all ideal preimages. 
\end{rem}

\begin{rem}
\label{rem:eliminating_full_paradigm}Let $u$ be a u-word , $\varphi$
an automorphism, $S$ a set of ideal preimages and $u\to\varphi\left(v\right)\in S$.
Let 
\[
S'=\left(S-\left\{ u\to\varphi\left(v\right)\right\} \right)\cup\left\{ u\to\varphi\left(v_{1}\right),\dots,u\to\varphi\left(v_{l}\right)\right\} 
\]
 such that $u\to\varphi\left(v_{1}\right),\dots,u\to\varphi\left(v_{l}\right)$
is a full paradigm of $u\to\varphi\left(v\right)$ then $S'$ is a
full set if and only if $S$ is a full set. 
\end{rem}

\begin{defn}
A full set is minimal if it does not contain a full paradigm.
\end{defn}

\begin{rem}
Any full set can be reduced to a minimal full set by repeating Remark
\ref{rem:eliminating_full_paradigm}.
\end{rem}

\begin{thm}
\label{thm:main1}Every u-word $u$ and automorphism $\varphi$ has
a unique finite minimal full set. Moreover it is the set of all minimal
ideal preimages. 
\end{thm}

Theorem \ref{thm:main_counting_words} follows directly. The multiset
$S_{u,\varphi}$ is constructed by taking the base words of the unique
minimal full set. We show first that it is unique then we will show
it exists. 
\begin{fact}
Let $v$ be a u-word let $\varphi$ an automorphism recall that we
denote the folding morphism by $\rho_{v}:\mathcal{F}_{\varphi}\left(v\right)\to\mathcal{\tilde{F}}_{\varphi}\left(v\right)$
. We observe that the graph $\mathcal{\tilde{F}}_{\varphi}\left(v\right)$
is a tree and that the subgraph $\varphi\left(v\right)$ is a line
segment . We observe that an edge $e$ in $\mathcal{\tilde{F}}_{\varphi}\left(v\right)$
belongs to the subgraph $\varphi\left(v\right)$ if and only if the
size of the set of preimages $\left|\rho_{v}^{-1}\left(e\right)\right|$
is odd.
\end{fact}

\begin{prop}
Let $\overline{v}$,$\overline{u}$ be u-words let $v$ be an orientation
of $\overline{v}$ s.t. $v=v_{1}\cdot v_{2}$, let $\overline{u}\to\mathcal{\tilde{F}}_{\varphi}\left(\overline{v}\right)$
and let $e$ be an edge contained in the image of $\overline{u}$.
We clarify the statement $\left|\rho_{v}^{-1}\left(e\right)\right|=\left|\rho_{v_{1}}^{-1}\left(e\right)\right|+\left|\rho_{v_{2}}^{-1}\left(e\right)\right|$
and show it holds. 
\end{prop}

\begin{proof}
First we claim that $\mathcal{\tilde{F}}_{\varphi}\left(\overline{v_{1}}\right)$
and $\tilde{\mathcal{F}}_{\varphi}\left(\overline{v_{2}}\right)$
both embed in $\tilde{\mathcal{F}}_{\varphi}\left(\overline{v}\right)$.
Because $\mathcal{\tilde{F}}_{\varphi}\left(\overline{v_{1}}\right)$
is a folded graph the function $\tilde{\mathcal{F}}_{\varphi}\left(\overline{v_{1}}\right)\to\mathcal{\tilde{F}}_{\varphi}\left(\overline{v}\right)$
is locally injective. Following the fact that both $\mathcal{\tilde{F}}_{\varphi}\left(\overline{v_{1}}\right)$
and $\mathcal{\tilde{F}}_{\varphi}\left(\overline{v}\right)$ are
trees a locally injective function must be globally injective. If
there exists an edge $e$ that satisfies 
\[
\xymatrix{e\ar[r]\ar[d] & u\ar[d]\\
\mathcal{\tilde{F}}_{\varphi}\left(\overline{v_{1}}\right)\ar[r] & \mathcal{\tilde{F}}_{\varphi}\left(\overline{v}\right)
}
\]
 it is unique because all the function in this square are injective.
By $e$ in $\left|\rho_{v_{1}}^{-1}\left(e\right)\right|$ we mean
this unique edge if it exists if it does not exist we consider $\left|\rho_{v_{1}}^{-1}\left(e\right)\right|=0$.
We notice that the graph $\mathcal{F}_{\varphi}\left(\overline{v_{1}\cdot v_{2}}\right)$
is a line segment obtained by gluing $\mathcal{F}_{\varphi}\left(\overline{v_{1}}\right)$
and $\mathcal{F}_{\varphi}\left(\overline{v_{2}}\right)$ along a
vertex. Therefore there are no edges in $\mathcal{F}_{\varphi}\left(\overline{v_{1}}\right)\cap\mathcal{F}_{\varphi}\left(\overline{v_{2}}\right)$.
Consequently we get 
\[
\left|\rho_{v}|_{\mathcal{F}_{\varphi}\left(\overline{v_{1}}\right)}^{-1}\left(e\right)\right|+\left|\rho_{v}|_{\mathcal{F}_{\varphi}\left(\overline{v_{2}}\right)}^{-1}\left(e\right)\right|=\left|\rho_{v}^{-1}\left(e\right)\right|
\]
Let $\mathcal{\tilde{F}}_{\varphi}\left(\overline{v_{1}}\right)\mathcal{\tilde{F}}_{\varphi}\left(\overline{v_{2}}\right)$
be the graph obtained by folding all vertexes in $\mathcal{F}_{\varphi}\left(\overline{v_{1}\cdot v_{2}}\right)$
besides the vertex connecting $\mathcal{F}_{\varphi}\left(\overline{v_{1}}\right)$
and $\mathcal{F}_{\varphi}\left(\overline{v_{2}}\right)$. The function
$\rho_{v_{1}\cdot v_{2}}$ factors through $\mathcal{\tilde{F}}_{\varphi}\left(\overline{v_{1}}\right)\mathcal{\tilde{F}}_{\varphi}\left(\overline{v_{2}}\right)$
because the order of folding does not effect the final folded graph.
The following diagram  
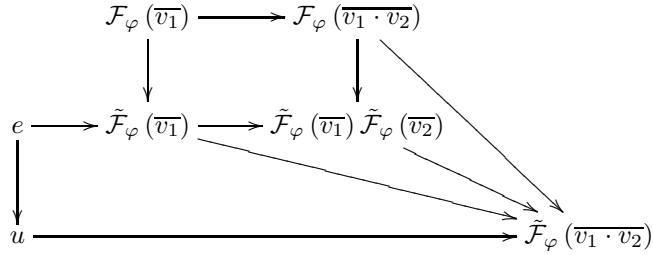
\[
\xymatrix{ & \mathcal{F}_{\varphi}\left(\overline{v_{1}}\right)\ar[r]\ar[d] & \mathcal{F}_{\varphi}\left(\overline{v_{1}\cdot v_{2}}\right)\ar[ddr]\ar[d]\\
e\ar[r]\ar[d] & \mathcal{\tilde{F}}_{\varphi}\left(\overline{v_{1}}\right)\ar[r]\ar[drr] & \mathcal{\tilde{F}}_{\varphi}\left(\overline{v_{1}}\right)\mathcal{\tilde{F}}_{\varphi}\left(\overline{v_{2}}\right)\ar[dr]\\
u\ar[rrr] &  &  & \mathcal{\tilde{F}}_{\varphi}\left(\overline{v_{1}\cdot v_{2}}\right)
}
\]
 commutes. Thus $\left|\rho_{v_{1}\cdot v_{2}}|_{\mathcal{F}_{\varphi}\left(\overline{v_{1}}\right)}^{-1}\left(e\right)\right|=\left|\rho_{v_{1}}^{-1}\left(e\right)\right|$.
\end{proof}
\begin{lem}
\label{lem:minimalPP}Let $\overline{u}\to\varphi\left(\overline{v}\right)$
be a minimal ideal-preimage and let $w$ be a cyclic word with $\overline{v}\to w$.
Let $\overline{u}\to\varphi\left(\overline{y}\right)$ be an ideal
preimage with $\overline{y}\to w$ such that 
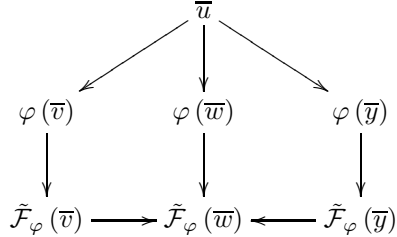
\[
\xymatrix{ & \overline{u}\ar[rd]\ar[d]\ar[ld]\\
\varphi\left(\overline{v}\right)\ar[d] & \varphi\left(\overline{w}\right)\ar[d] & \varphi\left(\overline{y}\right)\ar[d]\\
\mathcal{\tilde{F}}_{\varphi}\left(\overline{v}\right)\ar[r] & \mathcal{\tilde{F}}_{\varphi}\left(\overline{w}\right) & \mathcal{\tilde{F}}_{\varphi}\left(\overline{y}\right)\ar[l]
}
\]
 commutes. Then $\overline{u}\to\varphi\left(\overline{v}\right)$
is a sub ideal preimage of $\overline{u}\to\varphi\left(\overline{y}\right)$. 
\end{lem}

\begin{proof}
Assume by contradiction that it is not. Let $v,y$ be orientations
of $\overline{v},\overline{y}$ and let $w'$ be a cyclically reduced
representative of $w$. There are essentially 4 cases up to symmetry.
\begin{enumerate}
\item $y$ is a proper subword of $v$. 
\item Assume $l\left(v\right)<l\left(w'\right)$, write $w'=v\cdot w_{0}$
and consider $y$ to be $w_{0}$ or a subword of $w_{0}$. 
\item $y$ includes a subword both of $v$ and of $w_{0}$ . Write $v=v_{1}\cdot v_{2}$
and $w_{0}=w_{01}\cdot w_{02}$ and $y=v_{2}\cdot w_{01}$. 
\item Assume $l\left(v\right)\geq l\left(w\right)$ and $y$ is a cyclic
permutation of $v$ or a subword of a cyclic permutation of $w$.
\end{enumerate}
Case $1$: is in contradiction to the fact that $v$ is minimal. Case
2: There is graph morphism from $\overline{v\cdot w_{0}}$ the cyclic
word $w$.  Let $e$ be an edge in the image of $u\to\varphi\left(\overline{v\cdot w_{0}}\right)\to\mathcal{\tilde{F}}_{\varphi}\left(\overline{v\cdot w_{0}}\right)$.
Then $\left|\rho_{v\cdot w_{0}}^{-1}\left(e\right)\right|$ is odd.
Observe that $\left|\rho{}_{v}^{-1}\left(e\right)\right|$ is odd
because $v$ is an ideal preimage. Thus $\left|\rho_{w_{0}}^{-1}\left(e\right)\right|$is
even because 
\[
\left|\rho_{v\cdot w_{0}}^{-1}\left(e\right)\right|=\left|\rho{}_{v}^{-1}\left(e\right)\right|+\left|\rho{}_{w_{0}}^{-1}\left(e\right)\right|
\]
 Therefore there cannot be an ideal preimage $\overline{u}\to\varphi\left(\overline{w_{0}}\right)$.
Let $\overline{y}\to\overline{w_{0}}$. If $\overline{u}\to\varphi\left(\overline{y}\right)$
is an ideal preimage then so is $\overline{u}\to\varphi\left(\overline{w_{0}}\right)$
this a contradiction. Case 3: let $e$ be an edge in the image of
$\overline{u}$ then $\left|\rho_{v_{2}\cdot w_{01}}^{-1}\left(e\right)\right|$
is odd by assumption. We observe that 
\[
\left|\rho_{v\cdot w_{01}}^{-1}\left(e\right)\right|=\left|\rho{}_{v}^{-1}\left(e\right)\right|+\left|\rho{}_{w_{01}}^{-1}\left(e\right)\right|
\]
 because $\left|\rho_{v\cdot w_{01}}^{-1}\left(e\right)\right|$ and
$\left|\rho{}_{v}^{-1}\left(e\right)\right|$ are odd then $\left|\rho{}_{w_{01}}^{-1}\left(e\right)\right|$
is even. We observe that
\[
\left|\rho_{v_{2}\cdot w_{01}}^{-1}\left(e\right)\right|=\left|\rho{}_{v_{2}}^{-1}\left(e\right)\right|+\left|\rho{}_{w_{01}}^{-1}\left(e\right)\right|
\]
and $\left|\rho_{v_{2}\cdot w_{01}}^{-1}\left(e\right)\right|$ is
odd and $\left|\rho{}_{w_{01}}^{-1}\left(e\right)\right|$ is even
then $\left|\rho{}_{v_{2}}^{-1}\left(e\right)\right|$ is odd. We
chose $e$ arbitrarily thus for every edge $e$ in the image of $\overline{u}$
we get $\left|\rho{}_{v_{2}}^{-1}\left(e\right)\right|$ is odd. Therefore
there is a graph morphism $\overline{u}\to\varphi\left(\overline{v_{2}}\right)\to\mathcal{\tilde{F}}_{\varphi}\left(\overline{v_{2}}\right)$.
Because $\overline{u}\to\varphi\left(\overline{v}\right)$ is minimal
$\overline{u}\to\varphi\left(\overline{v_{2}}\right)$ is not an ideal
preimage thus there are words $x_{1}$ and $x_{2}$ with $x_{1}\cdot v_{2}\cdot x_{2}$
and an edge $e$ in the image of $\overline{u}$ s.t. $\left|\rho_{x_{1}\cdot v_{2}\cdot x_{2}}^{-1}\left(e\right)\right|$
is even. From the fact that 
\[
\left|\rho_{v\cdot x_{2}}^{-1}\left(e\right)\right|=\left|\rho{}_{v}^{-1}\left(e\right)\right|+\left|\rho{}_{x_{2}}^{-1}\left(e\right)\right|
\]
we get that $\left|\rho{}_{x_{2}}^{-1}\left(e\right)\right|$ is even
thus we get that $\left|\rho{}_{x_{1}\cdot v_{2}}^{-1}\left(e\right)\right|$
must be even as well. Finally we have 
\[
\left|\rho_{x_{1}\cdot v_{2}\cdot w_{01}}^{-1}\left(e\right)\right|=\left|\rho{}_{x_{1}\cdot v_{2}}^{-1}\left(e\right)\right|+\left|\rho{}_{w_{01}}^{-1}\left(e\right)\right|
\]
but we saw that $\left|\rho{}_{x_{1}\cdot v_{2}}^{-1}\left(e\right)\right|$
is even and $\left|\rho{}_{w_{01}}^{-1}\left(e\right)\right|$ is
even thus we get that $\left|\rho_{x_{1}\cdot v_{2}\cdot w_{01}}^{-1}\left(e\right)\right|$
is even and this is a contradiction to $v_{2}\cdot w_{01}$ being
an ideal preimage. Case 4: is proved in an almost identical way to
Case 3.
\end{proof}
\begin{lem}
Let $\overline{u}\to\varphi\left(\overline{v}\right)$ be a minimal
ideal-preimage and let $S$ be a finite full set that does not include
$\overline{u}\to\varphi\left(\overline{v}\right)$. Then $S$ is not
minimal i.e. it contains a full paradigm. 
\end{lem}

\begin{proof}
Let $m=\max\left\{ l\left(\overline{v}\right)|\overline{u}\to\varphi\left(\overline{v}\right)\in S\right\} $
the maximal length of the base words in the full set. Let $v$ be
an orientation of $\overline{v}$ an let 
\[
M=\left\{ \overline{x_{1}vx_{2}}\in F_{n}|\,x_{1}\cdot v\cdot x_{2},l\left(x_{1}v\right)=l\left(vx_{2}\right)=m\right\} 
\]
 a set of u-words. We notice that the u-words in $M$ inherit a structure
of $\varphi$ ideal preimages of $\overline{u}$ from the embedding
of $\overline{v}$. We construct a set of cyclic words $\overline{M}$
as following: If $w\in M$ is cyclically reduced we glue its ends;
If $w\in M$ is not cyclically reduced we add a letter to the end
before we glue. From the construction of $\overline{M}$ every cyclic
$w\in\overline{M}$ word has a graph morphism $\overline{v}\to w$.
Because $\overline{u}\to\varphi\left(\overline{v}\right)$ is an ideal
preimage every cyclic word $w\in\overline{M}$ has a morphism $\overline{u}\to\varphi\left(w\right)$
satisfying the appropriate diagram. From the second property of a
full set every $w\in\overline{M}$ has unique ideal preimage $\overline{u}\to\varphi\left(\overline{v_{w}}\right)$
with $\overline{v_{w}}\to w$ satisfying the correct diagram. We show
that the set $X=\left\{ \overline{u}\to\varphi\left(\overline{v_{w}}\right)|w\in\overline{M}\right\} \subset S$
contains a full paradigm. By Lemma \ref{lem:minimalPP} every $\overline{u}\to\varphi\left(\overline{v_{w}}\right)$
has $\overline{u}\to\varphi\left(\overline{v}\right)$ as a sub ideal
preimage. Particularly there is a graph morphism $\overline{v}\to\overline{v_{w}}\to w$.
By assumption $\Lambda\left(\overline{v_{w}}\right)\leq m$ and by
the construction $M$ and $\overline{M}$ there must be a $\overline{y}\in M$
such that $\overline{v}\to\overline{v_{w}}\to\overline{y}\to w$.
We now translate the problem to the language of Posets. We define
a Poset on the set $P=\left\{ \overline{x}|y\in M,\,\overline{v}\to\overline{x}\to\overline{y}\right\} $
(the composition $\overline{v}\to\overline{x}\to\overline{y}$ is
the original embedding of $\overline{v}$ in $\overline{y}$) by $x_{1}\leq x_{2}$
if and only if
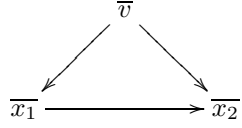
\[
\xymatrix{ & \overline{v}\ar[ld]\ar[rd]\\
\overline{x_{1}}\ar[rr] &  & \overline{x_{2}}
}
\]
 We notice that $\left\{ \overline{v_{w}}|\overline{u}\to\varphi\left(\overline{v_{w}}\right)\in X\right\} \subset P$
, $\overline{v}$ is the minimal element of $P$ and $M$ is the set
of maximal elements of $P$. As consequence of the second property
of a full set and the discussion above for all $y\in M$ there exists
a unique word $x\in X$ such that $x\leq y$. We define posets $T_{2}=\left\{ x|v\cdot x\right\} $
with $x_{1}\leq x_{2}$ if and only if $x_{1}$ is a prefix of $x_{2}$
and $T_{1}=\left\{ x|x\cdot v\right\} $ with $x_{1}\leq x_{2}$ if
and only if $x_{1}$ is a suffix of $x_{2}$. We define the order
$\left(x_{1},x_{2}\right)\leq\left(y_{1},y_{2}\right)$ if and only
if $x_{1}\leq y_{1}\wedge x_{2}\leq y_{2}$ on $T_{1}\times T_{2}$
and we notice that $P$ can be seen as a suborder of $T_{1}\times T_{2}$.
Let $T$ be the poset defined by a directed $2n-1$ regular tree such
that the root of $T$ is the minimal element then $T\cong T_{1}\cong T_{2}$.
A poset $T$ that is defined by a directed tree is characterized by
the following property: for all $a\in T$ the subposet $\left\{ b\in T|b\leq a\right\} $
is a total order. We denote the root, the minimal element, of $T$
by $r$ . We define $l:T\to\mathbb{N}$ by setting $l\left(a\right)$
to be the distance of $a$ from the root of the tree. We notice that
$l$ is equivalent to the length of a word in $T_{1}$ and $T_{2}$.
We say $a,b\in T$ are neighbors if there exists an element $c\in T$
such that $c<a\wedge c<b$ and $l\left(a\right)=l\left(b\right)=l\left(c\right)+1$.
In $T_{2}$ or $T_{1}$ two words are neighbors if they only differ
by the last or first letter. This motivates the next definition. Let
$a,b_{0}\in A$. A subset of $T\times T$ is said to be a full paradigm
if it has the form 
\[
\left\{ \left(a,b\right)\in T\times T|b_{0}<b,l\left(b\right)=l\left(b_{0}\right)+1\right\} 
\]
or 
\[
\left\{ \left(b,a\right)\in T\times T|b_{0}<b,l\left(b\right)=l\left(b_{0}\right)+1\right\} 
\]
We complete the translation by defining 
\[
P=\left\{ \left(a,b\right)\in T\times T|l\left(a\right)+l\left(b\right)\leq2m-l\left(v\right)\right\} 
\]
 and 
\[
M=\left\{ \left(a,b\right)\in T\times T|l\left(a\right)=l\left(b\right)=m-l\left(v\right)\right\} 
\]
 and $X\subset P$ is a set that satisfies for all $y\in M$ there
exists a unique element $x\in X$ such that $x\leq y$. We need to
show that if $\left(r,r\right)\nin X$ then $X$ contains a full paradigm. 

Let 
\[
X_{R}=\left\{ b\in T|\exists a\in T\,\left(a,b\right)\in X\right\} 
\]
 and let 
\[
X_{R}^{m}=\left\{ b\in X_{R}|\forall c\in X_{R}\:l\left(c\right)\leq l\left(b\right)\right\} 
\]
 the set of elements of maximal length in $X_{R}$. Let 
\[
X_{LX_{R}^{m}}=\left\{ a\in T|\exists b\in X_{R}^{m}\;\left(a,b\right)\in X\right\} 
\]
 Let $\left(a_{1},b_{1}\right)\in X$ such that $b_{1}\in X_{R}^{m}$
and $a_{1}\in X_{LX_{R}^{m}}$ is such that $l\left(a_{1}\right)$
is minimal among $X_{LX_{R}^{m}}$. Let $\oton b{2n-1}$ the neighbors
of $b_{1}$. Let 
\[
\left(a_{1}^{m},b_{2}^{m}\right)\in\left\{ \left(a,b\right)\in P|\left(a_{1},b_{2}\right)<\left(a,b\right)\right\} \cap M
\]
 then there is an element in $\left(x,y\right)\in X$ such that $\left(x,y\right)\leq\left(a_{1}^{m},b_{2}^{m}\right)$
i.e. $x\leq a_{1}^{m}$ and $y\leq b_{2}^{m}$. We recall that $\left\{ x\in T|x\leq a_{1}^{m}\right\} $
and $\left\{ y\in T|y\leq b_{2}^{m}\right\} $ are total orders because
$T$ is a tree. The element $b_{2}$ is a neighbor of $b_{1}$ therefore
it is maximal and $y\leq b_{2}$. By way of contradiction suppose
$y<b_{2}$ then $y<b_{1}$ because $b_{1}$ is a neighbor of $b_{2}$.
If $x\leq a_{1}$ then $\left(x,y\right)<\left(a_{1},b_{1}\right)$
this is a contradiction to the uniqueness in the definition of $X$.
If $x>a_{1}$ then $\left(x,y\right)<\left(x,b_{1}\right)$ and $\left(a_{1},b_{1}\right)<\left(x,b_{1}\right)$
also a contradiction to the uniqueness thus $y=b_{2}$. We chose $a_{1}$
to be minimal among $X_{LX_{R}^{m}}$ therefore $x\geq a_{1}$. If
$x=a_{1}$ and $\left(a_{1},b_{1}\right),\left(a_{1},b_{2}\right),\dots,\left(a_{1},b_{2n-1}\right)\in X$
then $X$ has a full paradigm and we are done. If not, without loss
of generality, we can assume $\left(a_{1},b_{2}\right)\nin X$ and
$a_{1}<x$. Let 
\[
X_{Lb_{2}}=\left\{ a\in T|\left(a,b_{2}\right)\in X\right\} 
\]
 and let $c_{1}\in X_{Lb_{2}}$an element such that $l\left(c_{1}\right)$
is maximal among $X_{Lb_{2}}$ then $\left(c_{1},b_{2}\right)\in X$
and $a_{1}<c_{1}$ . Let $c_{2}$ be a neighbor of $c_{1}$ and let
\[
\left(c_{2}^{m},b_{2}^{m}\right)\in\left\{ \left(a,b\right)\in P|\left(c_{2},b_{2}\right)<\left(a,b\right)\right\} \cap M
\]
then there is an element $\left(x,y\right)\in X$ such that $\left(x,y\right)\leq\left(c_{2}^{m},b_{2}^{m}\right)$
(this is a different $\left(x,y\right)$ then before). By way of contradiction
assume $y<b_{2}$ then $y<b_{1}$ because $b_{1}$ is a neighbor of
$b_{2}$. Because $c_{2}$ is a neighbor of $c_{1}$ and $a_{1}<c_{1}$
then $a_{1}<c_{2}$. Because $\left\{ x|x\leq c_{2}^{m}\right\} $
is a total order and $a_{1}\in\left\{ x|x\leq c_{2}^{m}\right\} $
then either $x<a_{1}$ or $x\geq a_{1}$. If $x<a_{1}$ then $\left(x,y\right)<\left(a_{1},b_{1}\right)$
and if $a_{1}\leq x$ then $\left(a_{1},b_{1}\right)\leq\left(x,b_{1}\right)$
and $\left(x,y\right)<\left(x,b_{1}\right)$ both contradict the uniqueness
condition. Thus $y\geq b_{2}$ but $b_{2}$ is maximal so $y=b_{2}$.
By way of contradiction assume $x<c_{2}$. Because $c_{2}$ is a neighbor
of $c_{1}$ then $x<c_{1}$ and we conclude $\left(x,b_{2}\right)<\left(c_{1},b_{2}\right)$
which is a contradiction to the uniqueness condition of $X$. Thus
$x\geq c_{2}$ but $c_{2}$ is maximal because it is a neighbor of
$c_{1}$ which is maximal so $x=c_{2}$. We get that $\left(c_{2},b_{2}\right)\in X$
but $c_{2}$ was an arbitrary neighbor of $c_{1}$ therefore for all
$c_{i}$ a neighbor of $c_{1}$ we have $\left(c_{i},b_{2}\right)\in X$
i.e. $X$ contains a full paradigm.
\end{proof}

We now show that every automorphism $\varphi$ and u-word $u$ has
a finite full set of ideal preimages. We call an automorphism $\varphi$
a good automorphism if every $u$ has a finite full set of $\varphi$
ideal preimages. We will show that if $\varphi$ and $\psi$ are good
automorphisms then $\varphi\circ\psi$ is also a good automorphism.
We conclude by showing that the Nielsen move $x_{1}\mapsto x_{1}x_{2},x_{j}\mapsto x_{j}$
for $j\neq1$ is a good automorphism by showing how to construct a
full set of ideal preimages for every $u$. It is trivial that the
other Nielsen moves are good. Thus a generating set of automorphisms
are good and all automorphisms generated by it. 
\begin{prop}
Let $\varphi=\psi_{1}\circ\psi_{2}$ be an automorphism such that
$\psi_{1},\psi_{2}$ are both good automorphisms then $\varphi$ is
a good automorphism.
\end{prop}

\begin{proof}
Let $u$ be a u-word then there is a $\psi_{1}$- full set $u\to\psi_{1}\left(v_{1}\right),\dots,u\to\psi_{1}\left(v_{l}\right)$
and every $v_{i}$ has a $\psi_{2}$ full set $v_{i}\to\psi_{2}\left(v_{i}^{1}\right),\dots,v_{i}\to\psi_{2}\left(v_{i}^{l_{i}}\right)$.
Let $w$ be a cyclic word with $v_{i}^{k}\to w$ then we have  \begin{equation} \label{eq:vik}
\begin{gathered}
\xymatrix{ & v_{i}\ar[ld]\ar[rd]\\ \psi_{2}\left(v_{i}^{k}\right)\ar[d] &  & \psi_{2}\left(w\right)\ar[d]\\ \tilde{\mathcal{F}}_{\psi_{2}}\left(v_{i}^{k}\right)\ar[rr] &  & \mathcal{\tilde{F}}_{\psi_{2}}\left(w\right) } 
\end{gathered} 
\end{equation}

From $v_{i}\to\psi_{2}\left(w\right)$ we get the diagram \begin{equation} \label{eq:vi}
\begin{gathered}
\xymatrix{ & u\ar[ld]\ar[rd]\\ \psi_{1}\left(v_{i}\right)\ar[d] &  & \psi_{1}\circ\psi_{2}\left(w\right)\ar[d]\\ \tilde{\mathcal{F}}_{\psi_{1}}\left(v_{i}\right)\ar[rr] &  & \mathcal{\tilde{F}}_{\psi_{1}}\left(\psi_{2}\left(w\right)\right) } 
\end{gathered} 
\end{equation}

From $v_{i}\to\psi_{2}\left(v_{i}^{k}\right)$ we get 
\[
\xymatrix{ & u\ar[ld]\ar[rd]\\
\psi_{1}\left(v_{i}\right)\ar[d] &  & \psi_{1}\circ\psi_{2}\left(v_{i}^{k}\right)\ar[d]\\
\mathcal{\tilde{F}}_{\psi_{1}}\left(v_{i}\right)\ar[rr] &  & \mathcal{\tilde{F}}_{\psi_{1}}\left(\psi_{2}\left(v_{i}^{k}\right)\right)
}
\]
 We combine these three diagrams to one diagram after applying the
functor $\tilde{\mathcal{F}}_{\psi_{1}}$ to the first diagram and
we get:
\[
\xymatrix{ & u\ar[ld]\ar[d]\ar[dr]\\
\psi_{1}\circ\psi_{2}\left(v_{i}^{k}\right)\ar[d] & \psi_{1}\left(v_{i}\right)\ar[d] & \psi_{1}\circ\psi_{2}\left(w\right)\ar[d]\\
\mathcal{\tilde{F}}_{\psi_{1}}\left(\psi_{2}\left(v_{i}^{k}\right)\right)\ar[d] & \tilde{\mathcal{F}}_{\psi_{1}}\left(v_{i}\right)\ar[r]\ar[l] & \mathcal{\tilde{F}}_{\psi_{1}}\left(\psi_{2}\left(w\right)\right)\ar[d]\\
\tilde{\mathcal{F}}_{\psi_{1}}\circ\tilde{\mathcal{F}}_{\psi_{2}}\left(v_{i}^{k}\right)\ar[rr] &  & \tilde{\mathcal{F}}_{\psi_{1}}\circ\mathcal{\tilde{F}}_{\psi_{2}}\left(w\right)
}
\]
We notice that generally $\tilde{\mathcal{F}}_{\psi_{1}}\circ\tilde{\mathcal{F}}_{\psi_{2}}\neq\tilde{\mathcal{F}}_{\psi_{1}\circ\psi_{2}}$
(see Proposition 3.9 in \cite{MR4187248}) thus the previous diagram
is not sufficient to show that $u\to\psi_{1}\circ\psi_{2}\left(v_{i}^{k}\right)$
is an ideal preimage. In order to prove that $u\to\psi_{1}\circ\psi_{2}\left(v_{i}^{k}\right)$
is ideal preimage we need three more commutative diagrams. For every
labeled graphs $w_{1}$ and $w_{2}$ with $w_{1}\to w_{2}$ there
is a commutative diagram 
\[
\xymatrix{\mathcal{\tilde{F}}_{\psi_{1}\circ\psi_{2}}\left(w_{1}\right)\ar[r]\ar[d] & \mathcal{\tilde{F}}_{\psi_{1}\circ\psi_{2}}\left(w_{2}\right)\ar[d]\\
\tilde{\mathcal{F}}_{\psi_{1}}\circ\tilde{\mathcal{F}}_{\psi_{2}}\left(w_{1}\right)\ar[r] & \tilde{\mathcal{F}}_{\psi_{1}}\circ\tilde{\mathcal{F}}_{\psi_{2}}\left(w_{2}\right)
}
\]
 For every u-word or cyclic word $w_{1}$ we have a commutative diagram
\[
\xymatrix{\psi_{1}\circ\psi_{2}\left(w_{1}\right)\ar[r]\ar[d] & \mathcal{\tilde{F}}_{\psi_{1}\circ\psi_{2}}\left(w_{1}\right)\ar[d]\\
\mathcal{\tilde{F}}_{\psi_{1}}\left(\psi_{2}\left(w_{1}\right)\right)\ar[r] & \mathcal{\tilde{F}}_{\psi_{1}}\circ\mathcal{\tilde{F}}_{\psi_{2}}\left(w_{1}\right)
}
\]
We combine all the commuting diagrams together to get Diagram \ref{eq:bigdiagram} 
\begin{equation}
\label{eq:bigdiagram}
\begin{gathered}
\begin{tikzcd}
& u \\ 
{\psi_{1}\circ\psi_{2}\left(v_{i}^{k}\right)} & {\psi_{1}\left(v_{i}\right)} & {\psi_{1}\circ\psi_{2}\left(w\right)} \\ 
{ \tilde{\mathcal{F}}_{\psi_{1}}\left(\psi_{2}\left(v_{i}^{k}\right)\right)} & {\tilde{\mathcal{F}}_{\psi_{1}}\left(v_{i}\right)} & {\tilde{\mathcal{F}}_{\psi_{1}}\left(\psi_{2}\left(w\right)\right)} \\ 
{\tilde{\mathcal{F}}_{\psi_{1}}\circ\tilde{\mathcal{F}}_{\psi_{2}}\left(v_{i}^{k}\right)} && {\tilde{\mathcal{F}}_{\psi_{1}}\circ\tilde{\mathcal{F}}_{\psi_{2}}\left(w\right)} \\ 
{\tilde{\mathcal{F}}_{\psi_{1}\circ\psi_{2}}\left(v_{i}^{k}\right)} && {\tilde{\mathcal{F}}_{\psi_{1}\circ\psi_{2}}\left(w\right)}
\arrow[from=1-2, to=2-1] 
\arrow[from=1-2, to=2-2] 
\arrow[from=1-2, to=2-3] 
\arrow[from=2-1, to=3-1] 
\arrow[from=2-3, to=3-3] 
\arrow[from=2-2, to=3-2] 
\arrow[from=3-2, to=3-1] 
\arrow[from=3-2, to=3-3] 	
\arrow[from=3-1, to=4-1] 	
\arrow[from=3-3, to=4-3] 	
\arrow[from=5-1, to=5-3] 	
\arrow[from=5-3, to=4-3] 	
\arrow[from=5-1, to=4-1] 	
\arrow[shift left=5, curve={height=-30pt}, from=2-3, to=5-3] 	\arrow[shift right=5, curve={height=30pt}, from=2-1, to=5-1] 	
\arrow[from=4-1, to=4-3] 
\end{tikzcd} 
\end{gathered}
\end{equation} Using all the commutative diagrams above one can observe that 
\begin{align*}
 & u\to\psi_{1}\circ\psi_{2}\left(v_{i}^{k}\right)\to\tilde{\mathcal{F}}_{\psi_{1}\circ\psi_{2}}\left(v_{i}^{k}\right)\to\tilde{\mathcal{F}}_{\psi_{1}\circ\psi_{2}}\left(w\right)\to\tilde{\mathcal{F}}_{\psi_{1}}\circ\tilde{\mathcal{F}}_{\psi_{2}}\left(w\right)\\
= & u\to\psi_{1}\circ\psi_{2}\left(w\right)\to\tilde{\mathcal{F}}_{\psi_{1}\circ\psi_{2}}\left(w\right)\to\tilde{\mathcal{F}}_{\psi_{1}}\circ\tilde{\mathcal{F}}_{\psi_{2}}\left(w\right)
\end{align*}
 But $\tilde{\mathcal{F}}_{\psi_{1}\circ\psi_{2}}\left(w\right)\to\tilde{\mathcal{F}}_{\psi_{1}}\circ\tilde{\mathcal{F}}_{\psi_{2}}\left(w\right)$
is injective because it is a locally injective function between trees.
We conclude that the diagram \begin{equation} \label{eq:vikPP}
\begin{gathered}
\xymatrix{ & u\ar[ld]\ar[rd]\\ \psi_{1}\circ\psi_{2}\left(v_{i}^{k}\right)\ar[d] &  & \psi_{1}\circ\psi_{2}\left(w\right)\ar[d]\\ \mathcal{\tilde{F}}_{\psi_{1}\circ\psi_{2}}\left(v_{i}^{k}\right)\ar[rr] &  & \mathcal{\tilde{F}}_{\psi_{1}\circ\psi_{2}}\left(w\right) } 
\end{gathered} 
\end{equation} is commutative. Thus $u\to\psi_{1}\circ\psi_{2}\left(v_{i}^{k}\right)$
is an ideal preimage. We want to show that $u\to\psi_{1}\circ\psi_{2}\left(v_{1}^{1}\right),\dots,u\to\psi_{1}\circ\psi_{2}\left(v_{1}^{l_{1}}\right),u\to\psi_{1}\circ\psi_{2}\left(v_{2}^{1}\right),\dots,u\to\psi_{1}\circ\psi_{2}\left(v_{l}^{l_{l}}\right)$
is a full set. Let $u\to\psi_{1}\circ\psi_{2}\left(w\right)$ then
because $u\to\psi_{1}\left(v_{1}\right),\dots,u\to\psi_{1}\left(v_{l}\right)$
is a full set there exists a $v_{i}\to\psi_{2}\left(w\right)$ satisfying
Diagram \ref{eq:vi}. Because $v_{i}\to\psi_{2}\left(v_{i}^{1}\right),\dots,v_{i}\to\psi_{2}\left(v_{i}^{l_{i}}\right)$
is a full set there exists $v_{i}^{k}\to w$ satisfying Diagram \ref{eq:vik}
thus for every $u\to\psi_{1}\circ\psi_{2}\left(w\right)$ there exists
ideal preimage $u\to\psi_{1}\circ\psi_{2}\left(v_{i}^{k}\right)$
with $v_{i}^{k}\to w$ that satisfies Diagram \ref{eq:vikPP}. Suppose
there is another $u\to\psi_{1}\circ\psi_{2}\left(v_{i'}^{k'}\right)$
with $v_{i'}^{k'}\to w$ satisfying Diagram \ref{eq:vikPP}. But $v_{i'}\to\psi_{2}\left(v_{i'}^{k'}\right)$
is an ideal preimage so $v_{i'}^{k'}\to w$ give us $v_{i'}\to\psi_{2}\left(w\right)$
satisfying Diagram \ref{eq:vik} and $v_{i'}\to\psi_{2}\left(w\right)$
give us $u\to\psi_{1}\circ\psi_{2}\left(w\right)$ satisfying Diagram
\ref{eq:vi}. But $u\to\psi_{1}\left(v_{1}\right),\dots,u\to\psi_{1}\left(v_{l}\right)$
a full set therefor there is a unique ideal preimage satisfying Diagram
\ref{eq:vi} so $v_{i'}=v_{i}$. In the same way $v_{i}\to\psi_{2}\left(v_{i}^{1}\right),\dots,v_{i}\to\psi_{2}\left(v_{i}^{l_{i}}\right)$
is a full set therefor there is a unique ideal preimage satisfying
Diagram \ref{eq:vik} so $v_{i'}^{k'}=v_{i}^{k}$. 
\end{proof}
\begin{lem}
\label{lem:Let--be} Let $\varphi$ be the automorphism of $F_{X}$
sending $x_{1}\mapsto x_{1}x_{2},x_{i}\mapsto x_{i}$ for $i\neq1$
(this is a Nielsen move). We show that every u-word $\overline{u}$
has a finite full set of ideal preimages 
\end{lem}

We notice that $\varphi$ satisfies the following special case of
Thurston's bounded cancellation lemma (see \cite{MR916179}). Let
$v_{1}$ and $v_{2}$ be words such that $v_{1}\cdot v_{2}$. Then
at most one letter can be canceled in the multiplication $\varphi\left(v_{1}\right)\varphi\left(v_{2}\right)$.
Thus to construct an ideal-preimages for a u-word $\overline{u}$
we need to change at most one letter in each affix of $\varphi^{-1}\left(\overline{u}\right)$.
Let $u$ be an orientation of $\overline{u}$ a suffix of $u$ and
a suffix of $u^{-1}$ are said to be pointing out and a prefix of
$u$ and prefix of $u^{-1}$ are said to be pointing in. We classify
8 different kinds of affixes pointing out. We classify the affixes
by examining the affix starting from the most outer letter looking
in until we find the first letter that is not $x_{2}$ or $x_{2}^{-1}$.
Theses affixes are the minimal subword of $w$ that allow us to know
the types of affixes of $\varphi\left(w\right)$. A word $u$ can
have multiple different ideal preimages. In Table \ref{tab:-pseudo-preimages}
they are either separated by commas or expressed by the letter $y$
standing in for different possible letters. We denote the ideal preimages
of $u$ by $\oton vl$. 
\begin{table}[ph]
\begin{centering}
\makebox[\textwidth][c]{%
\begin{tabular}{|c|c|c|c|c|c|}
\hline 
\# &  & suffix of $u$ & suffix of $\varphi^{-1}\left(u\right)$ & suffixes of $v_{i}$ & suffixes of $\varphi\left(v_{i}\right)$\tabularnewline
\hline 
1 & $z\neq x_{1},x_{2},x_{2}^{-1}$ & $z$ & $z$ & $z$ & $z$\tabularnewline
\hline 
2 &  & $x_{1}$ & $x_{1}x_{2}^{-1}$ & $x_{1}$ & $x_{1}x_{2}$\tabularnewline
\hline 
3 & $z\neq x_{1}$,$n\geq1$; $y\neq x_{1},x_{1}^{-1}$ & $zx_{2}^{n}$ & $zx_{2}^{n}$ & $zx_{2}^{n}x_{1},\,zx_{2}^{n}y$ & $zx_{2}^{n}x_{1}x_{2}$, $zx_{2}^{n}y$\tabularnewline
\hline 
4 & $y\neq x_{1},x_{2}^{-1},x_{1}^{-1}$ & $x_{1}x_{2}$ & $x_{1}$ & $x_{1}x_{1},\,x_{1}y$ & $x_{1}x_{2}x_{1}x_{2}$, $x_{1}x_{2}y$\tabularnewline
\hline 
5 & $n\geq2$; $y\neq x_{1},x_{1}^{-1}$ & $x_{1}x_{2}^{n}$ & $x_{1}x_{2}^{n-1}$ & $x_{1}x_{2}^{n-1}x_{1}$, $x_{1}x_{2}^{n-1}y$ & $x_{1}x_{2}^{n}x_{1}x_{2}$, $x_{1}x_{2}^{n}y$\tabularnewline
\hline 
6 & $z\neq x_{1}\,n\geq1$ & $zx_{2}^{-n}$ & $zx_{2}^{-n}$ & $zx_{2}^{-n},zx_{2}^{-n+1}x_{1}^{-1}$ & $zx_{2}^{-n},zx_{2}^{-n}x_{1}^{-1}$\tabularnewline
\hline 
7 & $n\geq1$ & $x_{1}x_{2}^{-n}$ & $x_{1}x_{2}^{-n-1}$ & $x_{1}x_{2}^{-n-1},x_{1}x_{2}^{-n}x_{1}^{-1}$ & $x_{1}x_{2}^{-n},x_{1}x_{2}^{-n}x_{1}^{-1}$\tabularnewline
\hline 
\end{tabular}}
\par\end{centering}
\caption{\label{tab:-pseudo-preimages}$\varphi$-ideal preimages}
\end{table}

\begin{example}
We construct the set of ideal preimages of $\overline{u}=\overline{x_{2}^{-1}x_{1}^{-1}x_{1}^{-1}x_{3}x_{1}}$
using Table \ref{tab:-pseudo-preimages}. We start by classifying
the two affixes of $\overline{u}$. The first affix is $x_{1}$. It
is of type $2$. The second is $x_{1}x_{2}$. It is type $4.$ We
calculate $\varphi^{-1}\left(\overline{u}\right)=\overline{x_{1}^{-1}x_{2}x_{1}^{-1}x_{3}x_{1}x_{2}^{-1}}$
we follow the table and we get the following set of ideal preimages
\begin{eqnarray*}
 & \overline{x_{1}^{-1}x_{1}^{-1}x_{2}x_{1}^{-1}x_{3}x_{1}},\:\overline{x_{2}^{-1}x_{1}^{-1}x_{2}x_{1}^{-1}x_{3}x_{1}},\\
 & \overline{x_{3}x_{1}^{-1}x_{2}x_{1}^{-1}x_{3}x_{1}},\:\overline{x_{3}^{-1}x_{1}^{-1}x_{2}x_{1}^{-1}x_{3}x_{1}},\dots,\:\overline{x_{n}^{-1}x_{1}^{-1}x_{2}x_{1}^{-1}x_{3}x_{1}}
\end{eqnarray*}
as for the morphism $\overline{u}\to\overline{\varphi\left(v_{i}\right)}$
we mark $\overline{u}$ in $\overline{\varphi\left(v_{i}\right)}$
by an underline
\begin{eqnarray*}
 & \overline{x_{2}^{-1}x_{1}^{-1}\underline{x_{2}^{-1}x_{1}^{-1}x_{1}^{-1}x_{3}x_{1}}x_{2}},\:\overline{x_{2}^{-1}\underline{x_{2}^{-1}x_{1}^{-1}x_{1}^{-1}x_{3}x_{1}}x_{2}},\\
 & \overline{x_{3}\underline{x_{2}^{-1}x_{1}^{-1}x_{1}^{-1}x_{3}x_{1}}x_{2}},\:\overline{x_{3}^{-1}\underline{x_{2}^{-1}x_{1}^{-1}x_{1}^{-1}x_{3}x_{1}}x_{2}},\dots,\:\overline{x_{n}^{-1}\underline{x_{2}^{-1}x_{1}^{-1}x_{1}^{-1}x_{3}x_{1}}x_{2}}
\end{eqnarray*}
\end{example}

\begin{rem}
Even though affixes were defined to be a proper subword, in these
definitions we allow for an affix to be the whole word. For instance
if $u=\overline{x_{1}x_{2}^{n}}$ we classify the two affixes of $u$
to be the following: $x_{1}x_{2}^{n}$ this is type $5$ and $x_{1}^{-1}$
pointing out this is type 2. Another limit case is the word is $u=\overline{x_{2}^{n}}$.
In this case we will not find a first letter that is not $x_{2}$
or $x_{2}^{-1}$. We will treat one affix as type $3$ and the other
as type 6. 
\end{rem}

To proof that $u\to\varphi\left(v_{i}\right)$ constructed using Table
\ref{tab:-pseudo-preimages} is an ideal preimage one needs to verify
that for every words $w_{1},w_{2}$ such that $w_{1}\cdot v_{i}\cdot w_{2}$
no letter of the subword $u$ in $\varphi\left(v_{i}\right)$ cancels
in the multiplication $\varphi\left(w_{1}\right)\varphi\left(v_{i}\right)\varphi\left(w_{2}\right)$.
This is straight forward verification because one just has to check
finite combinations of the affix types mentioned in Table \ref{tab:-pseudo-preimages}.
We have verified this but choose not to write it down because of its
length. We would like to show that a set of ideal preimages constructed
using Table \ref{tab:-pseudo-preimages} is a full set. We notice
that $\varphi^{-1}$ is the automorphism sending $x_{1}\mapsto x_{1}x_{2}^{-1},x_{i}\mapsto x_{i}$
for $i\neq1$. In order to compute $\varphi^{-1}$ ideal preimages
we can use Table \ref{tab:-pseudo-preimages} swapping $x_{2}$ with
$x_{2}^{-1}$. Denote by $\oton vl$ the set of $\varphi$-ideal preimages
of $u$ we obtain by using Table \ref{tab:-pseudo-preimages} and
denote by $v_{i}^{1},\dots,v_{i}^{l_{i}}$ the $\varphi^{-1}$-ideal
preimages of $v_{i}$ obtained using Table \ref{tab:-pseudo-preimages}.
In Table \ref{tab:inverse-pseudo-preimages} we calculate the suffixes
of $v_{i}^{k}$. 
\begin{table}[h]
\begin{centering}
\makebox[\textwidth][c]{%
\begin{tabular}{|c|c|c|c|c|}
\hline 
\# &  & suffix of $u$ & suffixes of $v_{i}$ & suffixes of $v_{i}^{k}$\tabularnewline
\hline 
1-2 & $z\neq x_{2},x_{2}^{-1}$ & $z$ & $z$ & $z$\tabularnewline
\hline 
\multirow{3}{0.5cm}{3} & \multirow{3}{2.5cm}{$z\neq x_{1}$, $n\geq1$, $y\neq x_{1},x_{1}^{-1},x_{2}$} & \multirow{3}{1cm}{$zx_{2}^{n}$} & $zx_{2}^{n}x_{1}$ & $zx_{2}^{n}x_{1}$\tabularnewline
\cline{4-5} \cline{5-5} 
 &  &  & $zx_{2}^{n}x_{2}$ & $zx_{2}^{n}x_{2}$,$zx_{2}^{n}x_{1}^{-1}$\tabularnewline
\cline{4-5} \cline{5-5} 
 &  &  & $zx_{2}^{n}y$ & $zx_{2}^{n}y$\tabularnewline
\hline 
\multirow{3}{0.5cm}{4} & \multirow{3}{2.5cm}{$y\neq x_{1},x_{2},x_{2}^{-1},x_{1}^{-1}$} & \multirow{3}{1cm}{$x_{1}x_{2}$} & $x_{1}x_{1}$ & $x_{1}x_{2}x_{1}$\tabularnewline
\cline{4-5} \cline{5-5} 
 &  &  & $x_{1}x_{2}$ & $x_{1}x_{2}x_{2},x_{1}x_{2}x_{1}^{-1}$\tabularnewline
\cline{4-5} \cline{5-5} 
 &  &  & $x_{1}y$ & $x_{1}x_{2}y$\tabularnewline
\hline 
\multirow{3}{0.5cm}{5} & \multirow{3}{2.5cm}{$n\geq2$, $y\neq x_{1},x_{1}^{-1},x_{2}$} & \multirow{3}{1cm}{$x_{1}x_{2}^{n}$} & $x_{1}x_{2}^{n-1}x_{1}$ & $x_{1}x_{2}^{n}x_{1}$\tabularnewline
\cline{4-5} \cline{5-5} 
 &  &  & $x_{1}x_{2}^{n-1}x_{2}$ & $x_{1}x_{2}^{n}x_{2},x_{1}x_{2}^{n}x_{1}^{-1}$\tabularnewline
\cline{4-5} \cline{5-5} 
 &  &  & $x_{1}x_{2}^{n-1}y$ & $x_{1}x_{2}^{n}y$\tabularnewline
\hline 
\multirow{2}{0.5cm}{6} & \multirow{2}{2.5cm}{$z\neq x_{1},n\geq1$, $y\neq x_{1}^{-1}$} & \multirow{2}{1cm}{$zx_{2}^{-n}$} & $zx_{2}^{-n+1}x_{1}^{-1}$ & $zx_{2}^{-n}x_{1}^{-1}$\tabularnewline
\cline{4-5} \cline{5-5} 
 &  &  & $zx_{2}^{-n}$ & $zx_{2}^{-n}y$\tabularnewline
\hline 
\multirow{2}{0.5cm}{7} & \multirow{2}{2.5cm}{$n\geq1$, $y\neq x_{1}^{-1}$} & \multirow{2}{1cm}{$x_{1}x_{2}^{-n}$} & $x_{1}x_{2}^{-n}x_{1}^{-1}$ & $x_{1}x_{2}^{-n}x_{1}^{-1}$\tabularnewline
\cline{4-5} \cline{5-5} 
 &  &  & $x_{1}x_{2}^{-n-1}$ & $x_{1}x_{2}^{-n}y$\tabularnewline
\hline 
\end{tabular}}
\par\end{centering}
\caption{\label{tab:inverse-pseudo-preimages}$\varphi^{-1}$-ideal preimages}
\end{table}

Because $v_{i}\to\varphi^{-1}\left(v_{i}^{k}\right)$ we get $u\to\varphi\circ\varphi^{-1}\left(v_{i}^{k}\right)=v_{i}^{k}$.
Examining Table \ref{tab:inverse-pseudo-preimages} we notice that
in cases 3-7 for every $v_{i}$ the suffixes of $v_{i}^{k}$ for every
$k$ form a full paradigm of the suffixes of $u$ and in cases 1-2
the suffix of $v_{i}^{k}$ is the suffix of $u$. Thus for every $u\to\varphi\left(w\right)$
there exists a $v_{i}^{k}$ s.t. $u\to v_{i}^{k}\to\varphi\left(w\right)$. 

Let $u\to\varphi\left(w\right)$ then there exists $u\to v_{i}^{k}\to\varphi\left(w\right)$.
From $v_{i}^{k}$ being an ideal preimage we get that $v_{i}\to\varphi^{-1}\circ\varphi\left(w\right)=w$.
If we plug these maps into Diagram \ref{eq:bigdiagram} we get that
$v_{i}\to w$ is the ideal preimage of $u\to\varphi\left(w\right)$.
This shows that for every $u\to\varphi\left(w\right)$ there is an
ideal-preimage in the set of ideal preimages constructed using Table
\ref{tab:-pseudo-preimages}.

Now we prove that it is unique. Examining Table \ref{tab:inverse-pseudo-preimages}
we notice that if $v_{i}\neq v_{j}$ for every $k_{1},k_{2}$ the
push-out of the square
\[
\xymatrix{u\ar[d]\ar[r] & v_{j}^{k_{2}}\ar[d]\\
v_{i}^{k_{1}}\ar[r] & v_{i}^{k_{1}}\underset{{\textstyle u}}{\coprod}v_{j}^{k_{2}}
}
\]
is a folded tree with at least one vertex of degree $3$. Let $u\to\varphi\left(w\right)$.
Assume by way of contradiction that both $v_{i}\to w$ and $v_{j}\to w$
are ideal preimages i.e. 
\[
\xymatrix{ & u\ar[rd]\ar[d]\ar[ld]\\
\varphi\left(v_{i}\right)\ar[d] & \varphi\left(w\right)\ar[d] & \varphi\left(v_{j}\right)\ar[d]\\
\tilde{\mathcal{F}}_{\varphi}\left(v_{i}\right)\ar[r] & \mathcal{\tilde{F}}_{\varphi}\left(w\right) & \tilde{\mathcal{F}}_{\varphi}\left(v_{j}\right)\ar[l]
}
\]
commutes. Because $v_{i}^{k}$ were constructed using Table \ref{tab:-pseudo-preimages}
there are $v_{i}^{k_{1}}\to\varphi\left(w\right)$ and $v_{j}^{k_{2}}\to\varphi\left(w\right)$
such that 
\[
\xymatrix{v_{i}\ar[d]\ar[dr] &  & v_{j}\ar[d]\ar[dl]\\
\varphi^{-1}\left(v_{i}^{k_{1}}\right)\ar[d] & w\ar[d] & \varphi^{-1}\left(v_{j}^{k_{2}}\right)\ar[d]\\
\tilde{\mathcal{F}}_{\varphi^{-1}}\left(v_{i}^{k_{1}}\right)\ar[r] & \mathcal{\tilde{F}}_{\varphi^{-1}}\left(\varphi\left(w\right)\right) & \tilde{\mathcal{F}}_{\varphi^{-1}}\left(v_{j}^{k_{2}}\right)\ar[l]
}
\]
commutes. But combining these diagrams and plugging into Diagram \ref{eq:bigdiagram}
we get the diagram
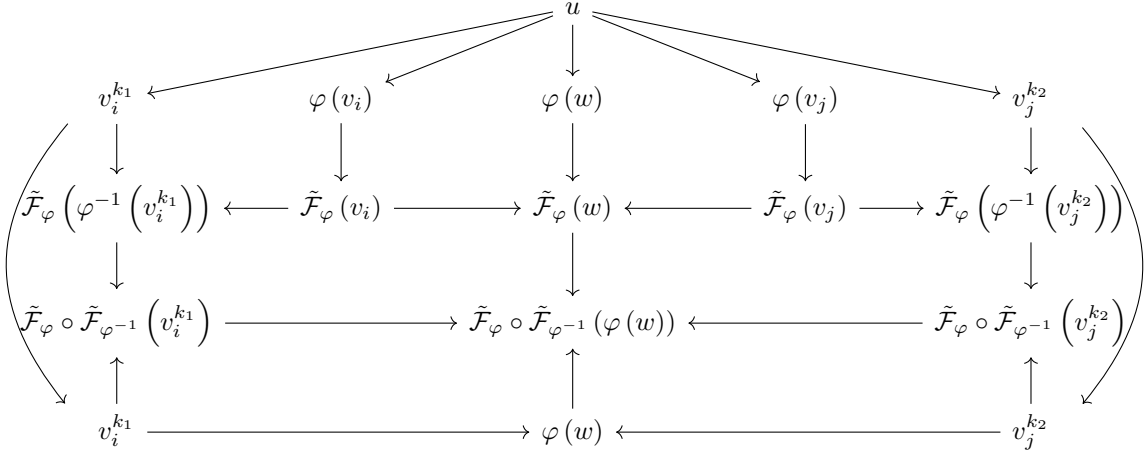
\begin{equation*}
\adjustbox{center}{%
\begin{tikzcd} 
&& u \\ 	{v_{i}^{k_{1}}} & {\varphi\left(v_{i}\right)} & {\varphi\left(w\right)} & {\varphi\left(v_{j}\right)} & {v_{j}^{k_{2}} } \\ 	{ \tilde{\mathcal{F}}_{\varphi}\left(\varphi^{-1}\left(v_{i}^{k_{1}}\right)\right)} & {\tilde{\mathcal{F}}_{\varphi}\left(v_{i}\right)} & {\tilde{\mathcal{F}}_{\varphi}\left(w\right)} & {\tilde{\mathcal{F}}_{\varphi}\left(v_{j}\right)} & { \tilde{\mathcal{F}}_{\varphi}\left(\varphi^{-1}\left(v_{j}^{k_{2}} \right)\right)} \\ 	{\tilde{\mathcal{F}}_{\varphi}\circ\tilde{\mathcal{F}}_{\varphi^{-1}}\left(v_{i}^{k_{1}}\right)} && {\tilde{\mathcal{F}}_{\varphi}\circ\tilde{\mathcal{F}}_{\varphi^{-1}}\left(\varphi\left(w\right)\right)} && {\tilde{\mathcal{F}}_{\varphi}\circ\tilde{\mathcal{F}}_{\varphi^{-1}}\left(v_{j}^{k_{2}} \right)} \\ 	{v_{i}^{k_{1}}} && {\varphi\left(w\right)} && {v_{j}^{k_{2}} } 	\arrow[from=2-1, to=3-1] 	\arrow[from=2-2, to=3-2] 	\arrow[from=3-2, to=3-1] 	\arrow[from=3-1, to=4-1] 	\arrow[from=5-1, to=5-3] 	\arrow[from=5-3, to=4-3] 	\arrow[from=5-1, to=4-1] 	\arrow[from=4-1, to=4-3] 	\arrow[from=1-3, to=2-1] 	\arrow[from=1-3, to=2-2] 	\arrow[from=1-3, to=2-3] 	\arrow[from=1-3, to=2-4] 	\arrow[from=1-3, to=2-5] 	\arrow[from=3-3, to=4-3] 	\arrow[from=2-3, to=3-3] 	\arrow[from=3-2, to=3-3] 	\arrow[shift right=5, curve={height=30pt}, from=2-1, to=5-1] 	\arrow[from=3-4, to=3-3] 	\arrow[from=3-4, to=3-5] 	\arrow[from=2-4, to=3-4] 	\arrow[from=2-5, to=3-5] 	\arrow[from=3-5, to=4-5] 	\arrow[from=4-5, to=4-3] 	\arrow[from=5-5, to=4-5] 	\arrow[from=5-5, to=5-3] 	\arrow[shift left=5, curve={height=-30pt}, from=2-5, to=5-5] 
\end{tikzcd}
 }
\end{equation*}

We conclude that the square 
\[
\xymatrix{u\ar[d]\ar[r] & v_{j}^{k_{2}}\ar[d]\\
v_{i}^{k_{1}}\ar[r] & \varphi\left(w\right)
}
\]
commutes. This implies a locally injective morphism $v_{i}^{k_{1}}\coprod_{u}v_{j}^{k_{2}}\to\varphi\left(w\right)$.
All vertices in $\varphi\left(w\right)$ are of degree $2$ but there
is a vertex of degree $3$ in $v_{i}^{k_{1}}\coprod_{u}v_{j}^{k_{2}}$
which is a contradiction. This concludes the proof of Lemma \ref{lem:Let--be}.

Let $u$ be an u-word an $\varphi$ automorphism. We conclude this
part with a method for calculating the unique minimal full set $S_{u,\varphi}$
(if we take the base words of the ideal preimages we get the multiset
mentioned in the introduction). First we decompose $\varphi$ into
Nielsen moves $\varphi=\varphi_{1}\circ\varphi_{2}\circ\cdots\circ\varphi_{l}$
by combining moves we can assume that every $\varphi_{k}$ is of the
form $x_{i}\mapsto x_{i}x_{j},x_{m}\mapsto x_{m}$ with $m\neq i$
or $x_{i}\mapsto x_{j}x_{i},x_{m}\mapsto x_{m}$ with $m\neq i$.
Now we use Table \ref{tab:-pseudo-preimages} to construct a full
set $S_{u,\varphi_{1}}$ of $\varphi_{1}$ ideal preimages of $u$
. Now for every $v\in S_{u,\varphi_{1}}$ we use Table to construct
a full set $S_{v,\varphi_{2}}$ of $\varphi_{2}$ ideal preimages
of $v$. We have shown that the set $\bigcup_{v\in S_{u,\varphi_{1}}}S_{v,\varphi_{2}}$
is a full set but generally it is not minimal. We can use Remark \ref{rem:eliminating_full_paradigm}
to replace $\bigcup_{v\in S_{u,\varphi_{1}}}S_{v,\varphi_{2}}$ by
the minimal set $S_{u,\varphi_{1}\circ\varphi_{2}}$ and now recursively
continue this process to $\varphi_{3}$ and onward. 

\section{Constructions of Inverse limit and Representation }

\subsection{Construction of inverse limit}

We know wish to create an inverse system of modules by defining $p_{k}:M_{k}\to M_{k-1}$
such that $p_{k}\pi_{k}=\pi_{k-1}$. We make two auxiliary definitions
\begin{defn}[U-word orientation]
 Let $f:\text{reduced-words}\to\text{u-words}$ the function that
forgets the orientation. A u-word orientation is a section $\sigma:\text{u-words}\to\text{reduced-words}$
s.t. $f\circ\sigma=\id$ choosing for every u-word an orientation.
Let $\sigma$ be an u-word orientation, let $\overline{u}$ be an
u-wrod then $\sigma\left(\overline{u}\right)=u$ or $\sigma\left(\overline{u}\right)=u^{-1}$.
Suppose $\sigma\left(\overline{u}\right)=u$ we say $u$ is the orientation
of $\overline{u}$. The word $u$ is also the orientation of $\overline{u^{-1}}$.
This is an auxiliary definition in order to make definitions precise.
We will mostly be interested in objects that are independent of this
choice. We choose an arbitrary u-word orientation we will use until
the end of the article.
\end{defn}

\begin{defn}[Intrinsic orientation]
  First we define by example then we give a precise definition.
Let $\overline{w}$ be a u-word with orientation $w$ and $\overline{u}$
a $k$-affix with orientation $u$. We say that $u$ is pointed inwards
if we write $\overline{w}$ as $\xymatrix{\cdot\ar[r]^{u} & \cdot\ar[r]^{v} & \cdot}
$ we say $u$ is pointed outward if write $\overline{w}$ as $\xymatrix{\cdot & \cdot\ar[r]^{v}\ar[l]_{u} & \cdot}
$ . We say that $\overline{u}$ is pointing inward in $\overline{w}$
if $u$ is a proper prefix of $w$ or $u^{-1}$ is a proper suffix
of $w$. We say $\overline{u}$ is pointing outwards in $\overline{w}$
if $u^{-1}$ is a proper prefix of $w$ or $u$ is a proper suffix
of $w$. We notice that the definition is independent of the orientation
of $\overline{w}$ but not independent of the orientation of $\overline{u}$.
Without orienting u-words we can not say that a $k$-affix $\overline{u}$
is pointing inward or outward but we can compare how a $k$-affix
$\overline{u}$ is situated in two different u-words. Let $\overline{w_{1}},\overline{w_{2}},\overline{u}$
be u-words s.t. $\overline{u}$ is a $k$-affix of $\overline{w_{1}}$
and also a $k$-affix of $\overline{w_{2}}$. We say that $\overline{u}$
has opposite intrinsic orientation in $\overline{w_{1}}$ and $\overline{w_{2}}$
if for every u-word orientation $\overline{u}$ is pointing inward
in one and outward in the other. We say that they have the same intrinsic
orientation in $\overline{w_{1}}$ and $\overline{w_{2}}$ if for
every u-word orientation they both point inward or both point outward.
 
\end{defn}

\begin{defn}
\label{def:p} We define two sequences of homomorphisms $\overleftarrow{p}_{k}:M_{k}\to M_{k-1}$
and $\overrightarrow{p}_{k}:M_{k}\to M_{k-1}$. We first make auxiliary
definition. Let $u$ be a length $k-1$ u-word and $w$ a length $k$
u-word we define the functionals $\overleftarrow{\delta}_{u},\overrightarrow{\delta}_{u}\in M_{k}^{*}\left(F_{n}\right)$
by 
\[
\overleftarrow{\delta}_{u}\left(w\right)=\begin{cases}
1 & \text{If \ensuremath{u} is \ensuremath{k-1} affix of \ensuremath{w} pointing outward}\\
0 & \otherwise
\end{cases}
\]
 
\[
\overrightarrow{\delta}_{u}\left(w\right)=\begin{cases}
1 & \text{If \ensuremath{u} is \ensuremath{k-1} affix of \ensuremath{w} pointing inward}\\
0 & \otherwise
\end{cases}
\]
 We note that it is impossible for a length $k$ word $w$ to have
two $k$-1-affixes $u$ pointing outwards. We identify $\hom\left(M_{k},M_{k-1}\right)$
with $M_{k}^{*}\otimes M_{k-1}$ and define $\overleftarrow{p}_{k},\overrightarrow{p}_{k}\in M_{k}^{*}\otimes M_{k-1}$
to be
\[
\overleftarrow{p}_{k}=\sum_{u\in W_{k-1}}\overleftarrow{\delta}_{u}\otimes u,\overrightarrow{p}_{k}=\sum_{u\in W_{k-1}}\overrightarrow{\delta}_{u}\otimes u
\]
\end{defn}

\begin{example}
We define an orientation on $W_{2}$ by $\sigma\left(\overline{xx}\right)=xx,\sigma\left(\overline{yy}\right)=yy,\sigma\left(\overline{xy}\right)=xy,\sigma\left(\overline{yx}\right)=yx,\sigma\left(\overline{y^{-1}x}\right)=y^{-1}x,\sigma\left(\overline{xy^{-1}}\right)=xy^{-1}$,
we look at $\overleftarrow{p}_{3}:M_{3}\to M_{2}$ in matrix form
\[
\begin{smallmatrix} & \overline{xxx} & \overline{yyy} & \overline{xxy} & \overline{xxy^{-1}} & \overline{yxx} & \overline{y^{-1}xx} & \overline{yyx} & \overline{yyx^{-1}} & \overline{xyy} & \overline{x^{-1}yy} & \overline{yxy} & \overline{y^{-1}xy^{-1}} & \overline{xyx} & \overline{x^{-1}yx^{-1}} & \overline{yxy^{-1}} & \overline{y^{-1}xy} & \overline{xyx^{-1}} & \overline{x^{-1}yx}\\
\overline{xx} & 1 & 0 & 1 & 1 & 0 & 0 & 0 & 0 & 0 & 0 & 0 & 0 & 0 & 0 & 0 & 0 & 0 & 0\\
\overline{yy} & 0 & 1 & 0 & 0 & 0 & 0 & 1 & 1 & 0 & 0 & 0 & 0 & 0 & 0 & 0 & 0 & 0 & 0\\
\overline{xy} & 0 & 0 & 0 & 0 & 0 & 0 & 0 & 0 & 1 & 0 & 0 & 0 & 1 & 0 & 0 & 0 & 1 & 0\\
\overline{yx} & 0 & 0 & 0 & 0 & 1 & 0 & 0 & 0 & 0 & 0 & 1 & 0 & 0 & 0 & 1 & 0 & 0 & 0\\
\overline{y^{-1}x} & 0 & 0 & 0 & 0 & 0 & 1 & 0 & 0 & 0 & 0 & 0 & 1 & 0 & 0 & 0 & 1 & 0 & 0\\
\overline{xy^{-1}} & 0 & 0 & 0 & 0 & 0 & 0 & 0 & 1 & 0 & 0 & 0 & 0 & 0 & 1 & 0 & 0 & 1 & 0
\end{smallmatrix}
\]
  Let $\tilde{M}_{k}=\im\pi_{k}$. 
\end{example}

\begin{thm}
\label{thm:mainsection2}We state properties of this construction.
\begin{enumerate}
\item $\tilde{M}_{k}=\im\pi_{k}=\text{Eq}\left(\overleftarrow{p}_{k},\overrightarrow{p}_{k}\right)\left(=\ker\left(\overleftarrow{p}_{k}-\overrightarrow{p}_{k}\right)\right)$ 
\item $\coker\left(\overleftarrow{p}_{k}-\overrightarrow{p}_{k}\right)\cong\nicefrac{\mathbb{Z}}{2\mathbb{Z}}$
i.e $\rank\tilde{M}_{k}\left(F_{n}\right)=\rank\ker\left(\overleftarrow{p}_{k}-\overrightarrow{p}_{k}\right)=\rank M_{k}\left(F_{n}\right)-\rank M_{k-1}\left(F_{n}\right)=n\left(2n-2\right)\left(2n-1\right)^{k-2}$
\item We denote $p_{k}=\overleftarrow{p}_{k}|_{\tilde{M}_{k}}=\overrightarrow{p}_{k}|_{\tilde{M}_{k}}$
then $p_{k}\pi_{k}=\pi_{k-1}$
\end{enumerate}
\end{thm}

Thus we get an homomorphism $\pi:\mathbb{Z}\left[C\right]\to\varprojlim\tilde{M}_{k}$

We begin by defining gluing. 
\begin{defn}[Gluing along a $k$-affix]
 \label{def:Gluing} Let $\overline{u},\overline{w}_{1},\overline{w}_{2}$
be u-words s.t. $\overline{u}$ is a $k$-affix situated in opposite
intrinsic orientation in $\overline{w}_{1}$ vis-\`{a}-vis $\overline{w}_{2}$
and let $v$ be a u-word s.t. both its $k$-affixes are $u$ situate
in opposite intrinsic orientation. The gluing of $w_{1}$ and $w_{2}$
along $u$ is the pushout of the square 
\[
\begin{array}{ccc}
u & \to & w_{1}\\
\downarrow\\
w_{2}
\end{array}
\]
 in the category of labeled graphs, the gluing of $v$ to itself along
$u$ is the co-equalizer of $u\rightrightarrows v$. Notice that because
$\overline{u}$ is of opposite intrinsic orientation in $\overline{w}_{1}$
and $\overline{w}_{2}$ the result of the pushout is an u-word if
$\overline{u}$ were to be situated in the same intrinsic orientation
we would have gotten a tripod. In the second case the result is a
cyclic word if $u$ had the same intrinsic orientation we would have
gotten a circle with a tail.
\end{defn}

\begin{example}
\label{exa:gluing} We continue Example \ref{exa:pi_k}. Recall
$\pi_{3}\left(xyx^{-1}y^{-1}\right)=\overline{xyx^{-1}}+\overline{yx^{-1}y^{-1}}+\overline{x^{-1}y^{-1}x}+\overline{y^{-1}xy}$.
We will show that we can reconstruct the cyclic word $xyx^{-1}y^{-1}$
by gluing the components $\pi_{3}\left(xyx^{-1}y^{-1}\right)$ along
$2$-affixes. We notice that $\overline{yx^{-1}}$ is situated in
opposite intrinsic orientation in $\overline{xyx^{-1}}$ vis-\`{a}-vis
$\overline{yx^{-1}y}$ thus we can glue $\overline{xyx^{-1}}$ and
$\overline{yx^{-1}y^{-1}}$ along $\overline{yx^{-1}}$ the result
is $\overline{xyx^{-1}y^{-1}}$ (Figure \ref{fig:examplegluing1}).
We continue and glue $\overline{xyx^{-1}y^{-1}}$ and $\overline{x^{-1}y^{-1}x}$
along $\overline{x^{-1}y^{-1}}$ and $\overline{xyx^{-1}y^{-1}x}$
with $\overline{y^{-1}xy}$ along $\overline{y^{-1}x}$ to get $\overline{xyx^{-1}y^{-1}xy}.$
Finally we glue $\overline{xyx^{-1}y^{-1}xy}$ to itself along $\overline{xy}$
to get the cyclic word $xyx^{-1}y^{-1}$(Figure \ref{fig:examplegluing2}).
\end{example}

\captionsetup[figure]
{%
name = Figure,
labelsep  = space  
}

\begin{figure}[ph]
 \centering     
\begin{minipage}{0.45\textwidth} 
	\adjustbox{ width=0.9\textwidth, center}{%
		\begin{tikzcd} 	
			\bullet & \bullet & \bullet & \bullet \\ 	& \bullet & \bullet & \bullet & \bullet \\ 	\bullet & \bullet & \bullet & 			\bullet & \bullet 	\arrow["x", from=1-1, to=1-2] 	\arrow["y", from=1-2, to=1-3] 	\arrow["x"', from=1-4, to=1-3] 			\arrow["y", from=2-2, to=2-3] 	\arrow["x"', from=2-4, to=2-3]    \arrow["y"', from=2-5, to=2-4]    
			\arrow["x", from=3-1, to=3-2] 	\arrow["y", from=3-2, to=3-3] 	\arrow["x"', from=3-4, to=3-3] 	
			\arrow["y"', from=3-5, to=3-4] 
		\end{tikzcd}
	}       
	\caption{\label{fig:examplegluing1}}     
\end{minipage}
\hfill     
\begin{minipage}{0.45\textwidth}
	\adjustbox{width=0.9\textwidth, center}{%
		\begin{tikzcd} 	
			\bullet &&& \bullet \\
			& \bullet & \bullet &&& \bullet & \bullet \\
			\bullet && \bullet & \bullet && \bullet & \bullet 	
			\arrow["x"', from=3-1, to=3-3] 	\arrow["x", from=2-2, to=2-3] 	
			\arrow["y", from=2-3, to=3-3] 	\arrow["y"', from=1-1, to=3-1] 	
			\arrow["x", from=1-1, to=1-4] 	\arrow["y", from=1-4, to=3-4] 	
			\arrow["x", from=2-6, to=2-7] 	\arrow["y", from=2-7, to=3-7] 	
			\arrow["x", from=3-6, to=3-7] 	\arrow["y", from=2-6, to=3-6] 
		\end{tikzcd}
		}       
	\caption{\label{fig:examplegluing2}}    
\end{minipage}
\end{figure}

 We would like to show that any vector in $\ker\left(\overleftarrow{p}_{k}-\overrightarrow{p}_{k}\right)$
can be lifted to a vector in $\mathbb{Z}\left[C\right]$. We do this
by generalizing the process showed in Example \ref{exa:gluing}. We
define a non-negative vector to be a linear combination of basis elements
such that all coefficients are non-negative. We think of a non-negative
vector in $M_{k}$ as a graph in the following way: A non negative
vector is written as a sum of $k$ length u-words we think of this
addition as disjoint union of the graphs thus a non-negative vector
is a disjoint union of $u$-words. We have a problem that gluing two
component along a $k-1$-affix results in a u-word that is longer
than $k$ thus it can not be represented in $M_{k}$. Therefore we
define $\hat{M}_{k}=\bigoplus_{j\geq k}M_{j}$ and the dual $\hat{M}_{k}^{*}=\prod_{j\geq k}M_{j}^{*}$.
We extend the homomorphism $\overleftarrow{p}_{k}-\overrightarrow{p}_{k}:M_{k}\to M_{k-1}$
to $\hat{M}_{k}$  by extending the auxiliary functionals in Definition
\ref{def:p}. Let $u$ be a u-word of length $k-1$ and $w$ an u-word
of length greater or equal to $k$. We notice that if the length of
$w$ is greater or equal to $2k-1$ then it is possible that both
$k-1$ affixes of $w$ are $u$ pointing outward. We extend the functional
$\overleftarrow{\delta}_{u}\in M_{k}^{*}$ to $\widehat{\overleftarrow{\delta}_{u}}\in\hat{M}_{k}^{*}$
in the following way 
\[
\widehat{\overleftarrow{\delta}_{u}}\left(w\right)=\begin{cases}
2 & \text{If both \ensuremath{k-1} affixes of \ensuremath{w} are \ensuremath{u} pointing outward}\\
1 & \text{If one \ensuremath{k-1} affix of \ensuremath{w} is \ensuremath{u} pointing outward}\\
0 & \otherwise
\end{cases}
\]
 we extend $\overrightarrow{\delta}_{u}$ analogously. Now we extend
$\widehat{\overleftarrow{p}_{k}}=\sum_{u\in W_{k-1}}\widehat{\overleftarrow{\delta}_{u}}\otimes u,\widehat{\overrightarrow{p}_{k}}=\sum_{u\in W_{k-1}}\widehat{\overrightarrow{\delta}_{u}}\otimes u$
finally $\widehat{\overleftarrow{p}_{k}-\overrightarrow{p}_{k}}:\hat{M}_{k}\to M_{k-1}$
defined to be $\widehat{\overleftarrow{p}_{k}-\overrightarrow{p}_{k}}=\widehat{\overleftarrow{p}_{k}}-\widehat{\overrightarrow{p}_{k}}$.
Clearly $\widehat{\overleftarrow{p}_{k}-\overrightarrow{p}_{k}}|_{M_{k}}=\overleftarrow{p}_{k}-\overrightarrow{p}_{k}$.
We define a morphism $\hat{\pi}_{k}:\hat{M}_{k}\to M_{k}$ analogous
to $\pi_{k}:\mathbb{Z}\left[C\right]\to M_{k}$. Let $w\in\hat{M}_{k}$
be an u-word of length greater or equal to $k$ (these words form
a basis for $\hat{M}_{k}$) and let $u$ be an u-word of length $k-1$.
We define $\hat{\pi}_{k}:\hat{M}_{k}\to M_{k}$ by $u^{*}\left(\hat{\pi}_{k}\left(w\right)\right)=\left|\hom\left(u,w\right)\right|$
we notice that $\hat{\pi}_{k}|_{M_{k}}=\id_{M_{k}}$ . We show lemmas
about gluing.  
\begin{lem}[Gluing Lemma]
 \label{Gluing-lemma-1} Let $\overline{w},\overline{v}$ be u-words
with $l\left(\overline{w}\right),l\left(\overline{v}\right)\geq k$,
let $\overline{t}$ be an $k-1$-affix situated in opposite intrinsic
orientation in $\overline{w}$ vis-\`{a}-vis $\overline{v}$, let
$\overline{u}$ be an u-word with $l\left(\overline{u}\right)\leq k$
and let $\overline{s}$ be the result of gluing $\overline{w}$ and
$\overline{v}$ along $\overline{t}$. We show that 
\[
\hom\left(\overline{u},\overline{s}\right)\cong\left(\hom\left(\overline{u},\overline{w}\right)\bigsqcup\hom\left(\overline{u},\overline{v}\right)\right)/\hom\left(\overline{u},\overline{t}\right)
\]
 thus 
\[
\left|\hom\left(\overline{u},\overline{s}\right)\right|=\left|\hom\left(\overline{u},\overline{w}\right)\right|+\left|\hom\left(\overline{u},\overline{v}\right)\right|-\left|\hom\left(\overline{u},\overline{t}\right)\right|
\]
.
\end{lem}

\begin{proof}
Because of general abstract nonsense considerations we get that there
is an injective morphism 
\[
\left(\hom\left(\overline{u},\overline{w}\right)\bigsqcup\hom\left(\overline{u},\overline{v}\right)\right)/\hom\left(\overline{u},\overline{t}\right)\hookrightarrow\hom\left(u,s\right)
\]
. We show it is surjective. Let $w$ be an orientation of $\overline{w}$,
$t$ and orientation of $\overline{t}$ and $s$ and orientation of
$\overline{s}$ such that we can write $w=w_{0}\cdot t$ , $v=t\cdot v_{0}$
and $s=w_{0}\cdot t\cdot v_{0}$. Assume by way of contradiction that
$u$ is a subword of $s$ that is not a subword of either $w$ or
$v$. This means it contains a non trivial segments both of $v_{0}$
and of $u_{0}$. We denote them by $v_{1}$ and $u_{1}$ and notice
$l\left(u_{1}\right),l\left(v_{1}\right)\geq1$. We can write $u=v_{1}\cdot t\cdot u_{1}$,
thus $l\left(u\right)=l\left(v_{1}\cdot t\cdot u_{1}\right)=l\left(v_{1}\right)+l\left(t\right)+l\left(u_{1}\right)=k-1+l\left(v_{1}\right)+l\left(u_{1}\right)\geq k+1$
and this is a contradiction since $l\left(u\right)\leq k$.
\end{proof}
\begin{cor}
Special case: If $l\left(\overline{u}\right)=k$ we get $\hom\left(\overline{u},\overline{t}\right)=\emptyset$.
Thus 
\[
\left|\hom\left(\overline{u},\overline{s}\right)\right|=\left|\hom\left(\overline{u},\overline{w}\right)\right|+\left|\hom\left(\overline{u},\overline{v}\right)\right|
\]
 
\end{cor}

\begin{prop}
\label{selfgluing}Let $\overline{w}$ be an u-word and let $\overline{t}$
be an u-word situated in $\overline{w}$ twice as two $k-1$-affixes
in opposite intrinsic orientation, let $v$ be the cyclic word obtained
by gluing $w$ to itself along $t$ and let $u$ be an u-word of length
$k$. Then $\hom\left(u,w\right)\cong\hom\left(u,v\right)$. 
\end{prop}

\begin{proof}
The morphism from $\overline{w}\to v$ is locally injective therefor
the morphism $\hom\left(u,w\right)\hookrightarrow\hom\left(u,v\right)$
is injective.We divide into two cases: 
\begin{enumerate}
\item $\frac{l\left(\overline{w}\right)}{2}>l\left(\overline{t}\right)$.
Let $w,t$ orientation of $\overline{w},\overline{t}$ and $v'$ a
cyclically reduced representative of $v$ such that we can write $w=t\cdot w_{0}\cdot t$
and $v'=t\cdot w_{0}$ . There are $l\left(w\right)-k+1$ subwords
of length $k$ in $w$ that is $l\left(t\cdot w_{0}\cdot t\right)-k+1=2l\left(t\right)+l\left(w_{0}\right)-k+1=l\left(t\right)+l\left(w_{0}\right)$
. There are $l\left(v\right)$ subwords of length $k$ in $v$ that
is $l\left(t\right)+l\left(w_{0}\right)$. Because $v$ and $\overline{w}$
have the same amount of subwords this means every subword of $v$
comes from a subword of $\overline{w}$ thus $\hom\left(u,w\right)\cong\hom\left(u,v\right)$. 
\item $\frac{l\left(\overline{w}\right)}{2}\leq l\left(\overline{t}\right).$
Let $w,t$ orientation of $\overline{w},\overline{t}$ such that we
can write $w$ in two ways $w=x\cdot t=t\cdot y$. This can be seen
as an equation over a free semigroup. The solutions of an equation
of this sort has the form $x=x_{0}y_{0},y=y_{0}x_{0}$, $t=\left(x_{0}y_{0}\right)^{r-1}x_{0}$,
$w=\left(x_{0}y_{0}\right)^{r-1}x_{0}y_{0}x_{0}$  with $r\in\mathbb{N}$
, the words $x_{0}$ and $y_{0}$ reduced words s.t. $x_{0}\cdot y_{0}$
and $y_{0}\cdot x_{0}$where $x_{0}\neq1$ but $y_{0}$ possibly trivial.
The cyclic word resulting from gluing $w$ to itself along $t$ is
$v=x_{0}y_{0}$. There are $l\left(w\right)-k+1$ subwords of length
$k$ in $w$ that is $l\left(t\cdot x\right)-k+1=l\left(t\right)+l\left(x\right)-k+1=l\left(x\right)$.
There are $l\left(v\right)$ subwords of length $k$ in $v$ that
is $l\left(v\right)=l\left(x_{0}y_{0}\right)=l\left(x\right)$. As
before every subword of $v$ comes from $w$ thus $\hom\left(u,w\right)\cong\hom\left(u,v\right)$, 
\end{enumerate}
\end{proof}
\begin{prop}
\label{swordtocyclic}Let $v$ be a cyclic word and let the u-word
$u$ be a subword of $v$ and let $k=l\left(u\right)+1$. Then there
exists an u-word $w$ with $u$ situated in it twice as two $k-1$-affixes
in opposite intrinsic orientation such that $v$ is the result of
gluing $w$ to itself along $u$. 
\end{prop}

\begin{proof}
We divide this into two cases:
\begin{enumerate}
\item $u$ is a proper subword i.e. $l\left(v\right)>l\left(u\right)$. 
\item $u$ is not a proper subword i.e. $l\left(u\right)\geq l\left(v\right)$ 
\end{enumerate}
\end{proof}
If $u$ is a proper subword than one can write $v=u\cdot v_{0}$ and
one can check that $w=u\cdot v_{0}\cdot u$. If $u$ is not a proper
subword than there exists $n\in\mathbb{N}$ a cyclically reduce representative
$v$ and a proper subword $v_{0}$ s.t. $u=v^{n}v_{0}$ then $w=v^{n+1}v_{0}$.
Let $v_{1}$be the proper subword s.t. $v=v_{0}v_{1}$ than $w=vv^{n-1}v_{0}v_{1}v_{0}$

\begin{prop}
\label{prop:homo}We prove propositions connecting the gluing lemmas
with the homomorphisms $\overleftarrow{p}_{k}-\overrightarrow{p}_{k},\hat{\pi}_{k},\widehat{\overleftarrow{p}_{k}-\overrightarrow{p}_{k}}$,
\begin{enumerate}
\item Let $v_{0},v_{1}\in\hat{M}_{k}$ be non negative vectors such that
$v_{1}$ is obtained by gluing two different components of $v_{0}$
along a common $k-1$-affix. Then $\hat{\pi}_{k}\left(v_{0}\right)=\hat{\pi}_{k}\left(v_{1}\right)$
and $\widehat{\overleftarrow{p}_{k}-\overrightarrow{p}_{k}}\left(v_{0}\right)=\widehat{\overleftarrow{p}_{k}-\overrightarrow{p}_{k}}\left(v_{1}\right)$.
\item Let $v\in\hat{M}_{k}$ be a non-negative vector and let $v_{0}\in M_{k}$
be the vector satisfying $\hat{\pi}_{k}\left(v\right)=v_{0}$. There
is a sequence of vectors $\ooton vn\in\hat{M}$ with $v_{n}=v$ such
that $v_{i+1}$ is obtained from $v_{i}$ by gluing two different
components of $v_{i}$ along a common $k-1$-affix.
\item $\left(\overleftarrow{p}_{k}-\overrightarrow{p}_{k}\right)\circ\hat{\pi}_{k}=\widehat{\overleftarrow{p}_{k}-\overrightarrow{p}_{k}}$
\end{enumerate}
\end{prop}

\begin{proof}
(1) The fact that $\hat{\pi}_{k}\left(v_{0}\right)=\hat{\pi}_{k}\left(v_{1}\right)$
follows from the gluing lemma (Lemma \ref{Gluing-lemma-1}). The fact
that $\widehat{\overleftarrow{p}_{k}-\overrightarrow{p}_{k}}\left(v_{0}\right)=\widehat{\overleftarrow{p}_{k}-\overrightarrow{p}_{k}}\left(v_{1}\right)$
follows because we glue along $k-1$ affixes situated in opposite
intrinsic orientation. (2) We show this for a single connected component
i.e. a word of length greater or equal to $k$ the general case follows.
Let $v$ be a $u$-word of length $k$ then $v_{0}=v$ and this is
the sequence. Let $v$ be a u-word of length $k+n$ and let $u_{1}$
be $k$-affix and $u_{2}$ a $k+n-1$-affix such that $u_{1}$ is
not a sub word of $u_{2}$ (i.e. these are affixes from opposite sides
of $v$). Then $u_{1},u_{2}$ have a common $k-1$-affix in opposite
orientation s.t. gluing along it gives $v$. Thus by the gluing lemma
(Lemma \ref{Gluing-lemma-1})
\[
\hat{\pi}_{k}\left(v\right)=\hat{\pi}_{k}\left(u_{1}\right)+\hat{\pi}_{k}\left(u_{2}\right)=u_{1}+\hat{\pi}_{k}\left(u_{2}\right)
\]
Thus we write $v_{n-1}=u_{1}+u_{2}$ and continue by induction to
decompose $u_{2}$. (3) Let $v\in\hat{M}$ be a non-negative vector,
$v_{0}=\hat{\pi}_{k}\left(v\right)$ and $v_{0},v_{1},\dots,v_{n}=v$
a sequence as in (2). Then 
\[
\left(\overleftarrow{p}_{k}-\overrightarrow{p}_{k}\right)\circ\hat{\pi}_{k}\left(v\right)=\overleftarrow{p}_{k}-\overrightarrow{p}_{k}\left(v_{0}\right)=\widehat{\overleftarrow{p}_{k}-\overrightarrow{p}_{k}}\left(v_{0}\right)=\widehat{\overleftarrow{p}_{k}-\overrightarrow{p}_{k}}\left(v_{i}\right)
\]
 for every $i$. Thus $\left(\overleftarrow{p}_{k}-\overrightarrow{p}_{k}\right)\circ\hat{\pi}_{k}\left(v\right)=\widehat{\overleftarrow{p}_{k}-\overrightarrow{p}_{k}}\left(v\right)$.
Lastly $\hat{M}_{k}$ is generated by non negative vectors therefor
$\left(\overleftarrow{p}_{k}-\overrightarrow{p}_{k}\right)\circ\hat{\pi}_{k}=\widehat{\overleftarrow{p}_{k}-\overrightarrow{p}_{k}}$. 
\end{proof}
\begin{prop}
\label{prop:pk-pk}We show claims about $\overleftarrow{p}_{k}-\overrightarrow{p}_{k}:M_{k}\to M_{k-1}$ 
\begin{enumerate}
\item Let $S=\left\{ v\in M_{k-1}|\sum_{u\in W_{k-1}}\left|u^{*}\left(v\right)\right|^{2}=2\right\} $
then $\left\langle S\right\rangle =\im\left(\overleftarrow{p}_{k}-\overrightarrow{p}_{k}\right).$
\item $\coker\left(\overleftarrow{p}_{k}-\overrightarrow{p}_{k}\right)\cong\nicefrac{\mathbb{Z}}{2\mathbb{Z}}$
(second property in Theorem \ref{thm:mainsection2})
\item Every $v\in\im\left(\overleftarrow{p}_{k}-\overrightarrow{p}_{k}\right)$
has a non negative preimage.
\item $\ker\overleftarrow{p}_{k}-\overrightarrow{p}_{k}$ is generated by
non-negative vectors.
\end{enumerate}
\end{prop}

\begin{proof}
(1) The set $S$ is the set of vectors that are linear combination
of exactly two length $k-1$ u-words with coefficients $1$ or $-1$.
It is clear that $\left\langle S\right\rangle \supset\im\left(\overleftarrow{p}_{k}-\overrightarrow{p}_{k}\right)$
because the u-words of length $k$ generate $M_{k}$ and every word
has exactly two $k-1$-affixes. In the other direction let $v\in S$
then there exists $\overline{u}_{1},\overline{u}_{2}\in W_{k-1}$
and $\epsilon_{1},\epsilon_{2}\in\left\{ 0,1\right\} $ such that
$v=\left(-1\right)^{\epsilon_{1}}\overline{u_{1}}+\left(-1\right)^{\epsilon_{2}}\overline{u_{2}}$.
Let $u_{1},u_{2}$ be the orientations of $\overline{u_{1}},\overline{u_{2}}$.
Then depending on the u-word orientation there exists $\epsilon_{3},\epsilon_{4}\in\left\{ -1,1\right\} $
and a letter $x$ such that there is not cancellation in $u_{1}^{\epsilon_{3}}xu_{2}^{\epsilon_{4}}$
and
\[
\left(-1\right)^{\epsilon_{1}}\overline{u_{1}}+\left(-1\right)^{\epsilon_{2}}\overline{u_{2}}=\widehat{\overleftarrow{p}_{k}-\overrightarrow{p}_{k}}\left(\overline{u_{1}^{\epsilon_{3}}xu_{2}^{\epsilon_{4}}}\right)=\left(\overleftarrow{p}_{k}-\overrightarrow{p}_{k}\right)\circ\hat{\pi}_{k}\left(\overline{u_{1}^{\epsilon_{3}}xu_{2}^{\epsilon_{4}}}\right)
\]
 second equality is (1) in Proposition \ref{prop:homo}. (2) We define
the functional $\lambda\in M_{k-1}^{*}$ by $\lambda\left(\sum_{u\in W_{k-1}}a_{u}u\right)=\sum_{u\in W_{k-1}}a_{u}$.
We notice that $\lambda\left(v\right)$ is even for every $v\in S$
. Let $u_{0}\in W_{k-1}$. Because $\lambda\left(u_{0}\right)=1$
we conclude $u_{0}\nin\left\langle S\right\rangle $. We show that
$S\cup\left\{ u_{0}\right\} $ generates $M_{k-1}$. If $u\in W_{k-1}$
then $u=u_{0}+\left(u-u_{0}\right)$ and $\left(u-u_{0}\right)\in S$.
We also notice that $2u_{0}=\left(u_{0}-u\right)+\left(u_{0}+u\right)\in\left\langle S\right\rangle $.
Thus $\nicefrac{\mathbb{Z}}{2\mathbb{Z}}\cong M_{k-1}/\im\left(\overleftarrow{p}_{k}-\overrightarrow{p}_{k}\right)$.
(3) Let $v\in\im\left(\overleftarrow{p}_{k}-\overrightarrow{p}_{k}\right)$
then by ($1$) there exists $\oton ul\in S$ such that $v=\sum c_{i}u_{i}$.
We notice that $S=-S$ thus we can choose $c_{i}$ to be positive.
We notice that for every $\epsilon_{3},\epsilon_{4},u_{1},u_{2}$
the vector $\hat{\pi}_{k}\left(\overline{u_{1}^{\epsilon_{3}}xu_{2}^{\epsilon_{4}}}\right)$
is non-negative thus every element in $S$ has a non negative preimage
so there exists non-negative vectors $\oton vl\in M_{k}$ such that
$\left(\overleftarrow{p}_{k}-\overrightarrow{p}_{k}\right)\left(v_{i}\right)=u_{i}$.
Thus $\left(\overleftarrow{p}_{k}-\overrightarrow{p}_{k}\right)\left(\sum c_{i}v_{i}\right)=v$
but $\sum c_{i}v_{i}$ is a non-negative combination of non-negative
vectors thus it is non-negative. (4) Let $w\in\ker\overleftarrow{p}_{k}-\overrightarrow{p}_{k}$
then we can write $w=w'-w''$ such that $w'$ and $w''$ are both
non negative. If $\overleftarrow{p}_{k}-\overrightarrow{p}_{k}\left(w'\right)=\overleftarrow{p}_{k}-\overrightarrow{p}_{k}\left(w''\right)=0$
the $w$ is a combination of non negative vectors. Assume $\overleftarrow{p}_{k}-\overrightarrow{p}_{k}\left(w'\right)\neq0$
then by (3) there exists a non negative preimage $u$ such that $\overleftarrow{p}_{k}-\overrightarrow{p}_{k}\left(-w'\right)=\overleftarrow{p}_{k}-\overrightarrow{p}_{k}\left(-w''\right)=\overleftarrow{p}_{k}-\overrightarrow{p}_{k}\left(u\right)$.
We write $v=\left(v'+u\right)-\left(u+v''\right)$ and notice that
$v'+u,v''+u$ are non-negative and $v'+u,v''+u\in\ker\overleftarrow{p}_{k}-\overrightarrow{p}_{k}$
.
\end{proof}
\begin{prop}
We conclude by proving the two remaining statements from Theorem \ref{thm:mainsection2}
\begin{enumerate}
\item $\im\pi_{k}=\tilde{M}_{k}\left(F_{n}\right)=\ker\overleftarrow{p}_{k}-\overrightarrow{p}_{k}$ 
\item We denote $p_{k}=\overleftarrow{p}_{k}|_{\tilde{M}_{k}}=\overrightarrow{p}_{k}|_{\tilde{M_{k}}}$
we show that $p_{k}:\tilde{M}_{k}\to\tilde{M}_{k-1}$ is an inverse
system such that $p_{k}\circ\pi_{k}=\pi_{k-1}$
\end{enumerate}
\end{prop}

\begin{proof}
(1) Let $w$ be a cyclic word. Choose an arbitrary subword $u$ of
$w$ of length $k-1$. By Proposition \ref{swordtocyclic} there is
an u-word $v$ s.t. if we glue $v$ to itself along $u$ we get $w$.
Then by Proposition \ref{selfgluing} we have $\pi_{k}\left(w\right)=\hat{\pi}_{k}\left(v\right)$.
Also clearly $\widehat{\overleftarrow{p}_{k}-\overrightarrow{p}_{k}}\left(v\right)=\overline{u}-\overline{u}=0$.
Thus 
\[
\left(\overleftarrow{p}_{k}-\overrightarrow{p}_{k}\right)\circ\pi_{k}\left(w\right)=\left(\overleftarrow{p}_{k}-\overrightarrow{p}_{k}\right)\circ\hat{\pi}_{k}\left(v\right)=\widehat{\overleftarrow{p}_{k}-\overrightarrow{p}_{k}}\left(v\right)=0
\]
 i.e. $\im\pi_{k}\subset\ker\overleftarrow{p}_{k}-\overrightarrow{p}_{k}$.
On the other hand let $v_{0}\in\ker\overleftarrow{p}_{k}-\overrightarrow{p}_{k}$
be a non-negative vector. By (4) in Proposition \ref{prop:pk-pk}
it is enough to show that $v_{0}\in\im\pi_{k}$ . If all the components
of $v$ are u-words that have a $k-1$-affix situated in them twice
in opposite intrinsic orientation than by Proposition \ref{selfgluing}
we can lift $v_{0}$ to $v\in\mathbb{Z}\left[C\right]$ s.t. $\pi_{k}\left(v\right)=v_{0}$
by gluing each component to itself along its $k-1$-affix. If not
let $w$ be a component of $v$ that does not have the above property
and let $u$ be a $k-1$-affix of $w$. Because $\overleftarrow{p}_{k}-\overrightarrow{p}_{k}\left(v\right)=0$
there exists another component of $v$ with $u$ situated in it in
intrinsic orientation opposite to $w$ we denote it by $t$. Let $v_{1}$
be the vector in $\hat{M}_{k}$ obtained by gluing $w$ and $t$ along
$u$. By (1) in Proposition \ref{prop:homo} $\hat{\pi}_{k}\left(v_{1}\right)=v_{0}$
and $\widehat{\overleftarrow{p}_{k}-\overrightarrow{p}_{k}}\left(v_{1}\right)=0$.
Now we can recursively apply the previous step until we get $v_{n}$
s.t. all its components are u-words that have the same $k-1$-affix
situated in them twice in opposite intrinsic orientation. The vector
$v_{n}$ can be lifted to a vector $v\in\mathbb{Z}\left[C\right]$
s.t. $\pi_{k}\left(v\right)=\hat{\pi}_{k}\left(v_{n}\right)=v_{0}$
by gluing and Proposition  \ref{selfgluing}. (2) We show that $\hat{\pi}_{k-1}|_{M_{k}}=\overleftarrow{p}_{k}+\overrightarrow{p}_{k}$.
Let $v$ be a length $k$ u-word then it has exactly two $k-1$-subwords
moreover these are its $k-1$-affixes therefore $\hat{\pi}_{k-1}\left(v\right)=\left(\overleftarrow{p}_{k}+\overrightarrow{p}_{k}\right)\left(v\right)$.
We show that $\hat{\pi}_{k-1}\circ\pi_{k}=2\pi_{k-1}$. If $v$ is
a cyclic word of length $1$ then $\pi_{k}\left(v\right)=\overline{v^{k}},\hat{\pi}_{k-1}\left(\overline{v^{k}}\right)=2\overline{v^{k-1}}=2\pi_{k-1}\left(v\right)$.
Let $v$ be a cyclic word of length greater then $1$. Let $u$ be
a subword of length $k-1$ with a graph morphism $u\to v$ then there
are two different length $k$ subwords $w_{1},w_{2}$ such that $u\to w_{1}\to v,u\to w_{2}\to v$.
We get $w_{1}$ by adding the subsequent letter to $u$ in $v$ on
one side of $u$ and get $w_{2}$ by adding the subsequent letter
to $u$ in $v$ on the other side of $u$. This shows that for every
$u\to v$ there are two different morphisms $u\to_{1}\pi_{k}\left(v\right)$
and $u\to_{2}\pi_{k}\left(v\right)$. But every $u$ that is a subword
of $\pi_{k}\left(v\right)$ is also a subword of $v$. This shows
that $\hat{\pi}_{k-1}\circ\pi_{k}\left(v\right)=2\pi_{k-1}\left(v\right)$.
  We notice that $\hat{\pi}_{k-1}|_{\tilde{M}_{k}}=\overleftarrow{p}_{k}|_{\tilde{M}_{k}}+\overrightarrow{p}_{k}|_{\tilde{M}_{k}}=2p_{k}$
and because $\tilde{M}_{k}=\im\pi_{k}$ we get 
\[
2\pi_{k-1}=\hat{\pi}_{k-1}\circ\pi_{k}=\hat{\pi}_{k-1}|_{\tilde{M}_{k}}\circ\pi_{k}=2p_{k}\circ\pi_{k}
\]
 thus $p_{k}\circ\pi_{k}=\pi_{k-1}$.  
\end{proof}
In conclusion we have a homomorphism $\pi:\mathbb{Z}\left[C\right]\to\varprojlim\tilde{M}_{k}$.

\subsection{Properties of $\pi$}

We now try to understand the kernel of $\pi$. 
\begin{prop}
Let $w$ be a cyclically reduced word. We look at $w$ and $w^{l}$
as cyclic words then $l\pi_{k}\left(w\right)=\pi_{k}\left(w^{l}\right)$
for every $k$. Thus 
\[
\left\langle lw-w^{l}|w\text{ cyclicaly reduced},l\in\mathbb{N}\right\rangle <\ker\pi
\]
. 
\end{prop}

\begin{proof}
We notice that the morphism $w^{l}\to w$ is an $l$-cover space of
the cyclic word $w$ we also notice that a u-word is simply connected.
Let $u$ be a u-word any morphism $u\to w$ has exactly $l$ different
lifts to $w^{l}$ thus 
\[
\left|\hom\left(u,w^{l}\right)\right|=l\left|\hom\left(u,w\right)\right|
\]
. Thus we get $l\pi_{k}\left(w\right)=\pi_{k}\left(w^{l}\right)$
for every $k$ and $l$.
\end{proof}
We would like to show that $\ker\pi=\left\langle lw-w^{l}|w\text{ cyclicaly reduced},l\in\mathbb{N}\right\rangle $
for this we need a lemma with a somewhat technical proof.
\begin{lem}
\label{lem:notPower}Let $w$ be a cyclically reduced word. The word
$w$ is a nontrivial power if and only if for every cyclic permutation
$w'$ there are words $t$ and $v$ such that $w'=t\cdot v\cdot t$
where t and $v$ are proper subwords. $t\neq1$ but $v$ may be trivial. 
\end{lem}

\begin{proof}
One direction is easy: If $w=u^{k}$ with $k\geq2$ it can be decomposed
as $uu^{k-2}u$. A cyclic permutation of a power is also a power and
thus it satisfies the condition. Suppose $w$ satisfies the condition.
We say a decomposition $t\cdot v\cdot t$ of a cyclic permutation
of $w$ is reducible if $t$ has the form $t=s\cdot t_{0}\cdot s$
where $s$ and $t_{0}$ are proper subwords ($s\neq1$, $t_{0}$ may
be trivial) and irreducible otherwise. Let $t\cdot v\cdot t$ be an
irreducible decomposition of a cyclic permutation of $w$ such that
$l\left(t\right)$ is maximal among all irreducible decompositions
of cyclic permutations of $w$. We want to show that $w$ has a cyclic
permutation that is equal to $t^{k}$ for some $k$. If $v=1$ then
$w$ is a power, therefore assume $v\neq1$. Let us examine the cyclic
permutation $v\cdot t\cdot t$. Because $w$ satisfies the condition
in the lemma $v\cdot t\cdot t$ has an irreducible decomposition $s\cdot u\cdot s$.
Because $t$ is maximal this means that $s$ is a proper suffix of
$t$ or $s=t$. We want to show that $s=t$.
\begin{rem*}
If we show that $t$ has proper suffix $a$ that is also a proper
prefix we will reach a contradiction. For $l\left(a\right)<\frac{l\left(t\right)}{2}$
this is clear because we can write $t=a\cdot t_{0}\cdot a$ i.e. $t\cdot v\cdot t$
is reducible. For $l\left(a\right)\geq\frac{l\left(t\right)}{2}$.
We write $t=t_{1}\cdot a=a\cdot t_{2}$ this the same kind of equation
over a free semigroup as in Proposition  \ref{selfgluing}. As before
we have a family of solutions $t_{1}=xy,t_{2}=yx,a=\left(xy\right)^{r-1}x$
with $x,y$ reduced words such that $x\cdot y$ and $y\cdot x$ and
$r\in\mathbb{N}$ with $x\neq1$ and $y$ possibly trivial. Thus $t=at_{2}=\left(xy\right)^{r-1}xyx=x\left(y\left(xy\right)^{r-1}\right)x$
which is a contradiction to $tvt$ being an irreducible decomposition.
 
\end{rem*}
Assume $l\left(s\right)<l\left(t\right)$. Thus $s$ is a suffix of
$t$ and we write $t=t_{1}\cdot s$ . We divide into three different
cases
\begin{enumerate}
\item $l\left(s\right)>l\left(v\right)$: We write $s=s_{1}\cdot s_{2}$
such that $s_{1}=v$ and $t=s_{2}\cdot t_{2}$. Thus $s_{2}\cdot t_{2}=t=t_{1}\cdot s=t_{1}s_{1}\cdot s_{2}$
. We see that $s_{2}$ is a both proper prefix and a proper suffix
of $t$ this is a contradiction.  
\item $l\left(s\right)=l\left(v\right)$: Then $s=v$ and we consider the
cyclic permutation $t\cdot t\cdot s$. It has an irreducible decomposition
$s_{1}\cdot u_{1}\cdot s_{1}$. We notice $l\left(s_{1}\right)\leq l\left(t\right)$
because $t$ is maximal. We write $t=s_{1}\cdot t_{2}$ ($s_{1}\neq1$
but $t_{2}$ is possibly trivial) 
\begin{enumerate}
\item If $l\left(s_{1}\right)\leq l\left(s\right)<l\left(t\right)$ then
$s_{1}$ is a prefix of $t$ but it is also a suffix of $s$ (or equal
to s) thus a suffix of $t$ this is a contradiction.
\item If $l\left(s_{1}\right)>l\left(s\right)$ then we write $s_{1}=s_{11}\cdot s_{12}$
, $t=t_{3}\cdot s_{11}$, $s_{12}=s$. We get $t_{3}\cdot s_{11}=t=s_{1}\cdot t_{2}=s_{11}\cdot s_{12}\cdot t_{2}$.
Thus $s_{11}$ is both a proper prefix and a proper suffix of $t$
this is a contradiction. (This still holds of $t_{2}=1$)
\end{enumerate}
\item $l\left(s\right)<l\left(v\right)$: We write $v=s\cdot v_{1}$ and
$v\cdot t\cdot t=s\cdot v_{1}\cdot t\cdot t$. We denote $s=\hat{s}_{1}$
and proceed by induction. 
\end{enumerate}
We examine the permutation $v_{k}\cdot t\cdot t\cdot\hat{s}_{k}$.
The subword $\hat{s}_{k}$ satisfies the following properties: (1)
It is a concatenation of suffixes of $t$. (2) If $r$ is a suffix
of $\hat{s}_{k}$ or $\hat{s}_{k}$ itself and it satisfies $l\left(r\right)>l\left(s\right)$
then it is a concatenation of two or more suffixes of $t$. We notice
that $\hat{s}_{1}=s$ satisfies properties (1) and (2). If $u$ satisfies
properties $(1)$ and (2) then any suffix of $u$ also satisfies these
properties. If $u_{1}$ and $u_{2}$ both satisfy properties (1) and
(2) then there concatenation $u_{1}\cdot u_{2}$ satisfies the properties
as well.  The cyclic permutation $v_{k}\cdot t\cdot t\cdot\hat{s}_{k}$
has an irreducible decomposition $s_{k}\cdot u_{k}\cdot s_{k}$. Because
$t$ is maximal $l\left(t\cdot\hat{s}_{k}\right)>l\left(s_{k}\right)$
this means $s_{k}$ is suffix of $t\cdot\hat{s}_{k}$. We notice that
$s_{k}$ satisfies properties (1) and (2). We divide into three cases: 
\begin{enumerate}
\item $l\left(s_{k}\right)>l\left(v_{k}\right)$: Because $t$ is maximal
then $l\left(s_{k}\right)<l\left(v_{k}\cdot t\right)$ thus there
is suffix of $s_{k}$ that is a prefix of $t$ denote it by $s_{k1}$.
The subword $s_{k1}$ satisfies properties (1) and (2). Thus $s_{k1}$
has a prefix that is a suffix of $t$ or it is itself a suffix of
$t$. Either way $t$ has a prefix that is also its suffix and this
is a contradiction.
\item If $l\left(s_{k}\right)=l\left(v_{k}\right)$ then $v_{k}\cdot t\cdot t\cdot\hat{s}_{k}=s_{k}\cdot t\cdot t\cdot\hat{s}_{k}$
then there is a cyclic permutation $t\cdot t\cdot\hat{s}_{k}\cdot s_{k}$
. We notice $\hat{s}_{k}s_{k}$ satisfies properties (1) and (2).
We write $\hat{s}_{k}s_{k}=\hat{s}'_{k}$ and examine $tt\hat{s}'_{k}.$
The cyclic permutation $tt\hat{s}'_{k}$ has a irreducible decomposition
$s'_{k}\cdot u'_{k}\cdot s'_{k}$ we notice that $l\left(s'_{k}\right)<l\left(t\cdot\hat{s}'_{k}\right)$
thus $s'_{k}$ is a suffix of $t\cdot\hat{s}'_{k}$ so it satisfies
properties (1) and (2). There are two options 
\begin{enumerate}
\item The subword $s'_{k}$ is not a concatenation of two or more suffixes
of $t$. Because of property (2), it is a suffix of $t$ satisfying
$l\left(s'_{k}\right)\leq l\left(s\right)$ and by assumption $l\left(s'_{k}\right)\leq l\left(s\right)<l\left(t\right)$
. Thus $s'_{k}$ is both a prefix and a suffix of $t$ and this is
a contradiction.
\item Otherwise observe that $l\left(s'_{k}\right)\leq l\left(t\right)$.
Thus a prefix of $s'_{k}$ is also a prefix of $t$. Because $s'_{k}$
is a concatenation of suffixes of $t$ it has a prefix that is a suffix
of $t$ and this is a contradiction. 
\end{enumerate}
\item If $l\left(s_{k}\right)<l\left(v_{k}\right)$ then we write $v_{k}=s_{k}v_{k+1}$
and $v_{k}\cdot t\cdot t\cdot\hat{s}_{k}=s_{k}\cdot v_{k+1}\cdot t\cdot t\cdot\hat{s}_{k}.$
We examine the cyclic permutation $v_{k+1}\cdot t\cdot t\cdot\hat{s}_{k}\cdot s_{k}$.
We notice that $\hat{s}_{k}\cdot s_{k}$ satisfies properties (1)
and (2) and we write $\hat{s}_{k+1}=\hat{s}_{k}\cdot s_{k}$. We notice
that $l\left(v_{k+1}\right)<l\left(v_{k}\right)$ thus case $3$ cannot
repeat indefinitely and ultimately we reach a contradiction. 
\end{enumerate}
We close the argument with another induction. We had a cyclic permutation
$v\cdot t\cdot t$ that has a decomposition $s\cdot u\cdot s$ and
we concluded that $s=t$ . Because $t$ is irreducible then $l\left(v\right)\geq l\left(t\right)$.
If $v=t$ we are done if not we write $v=tv_{1}$ and we observe the
cyclic permutation $v_{1}t^{1}tt$. This is the base for induction.
Let $v_{k}t^{k}tt$ be a cyclic permutation. Then writing $v'_{k}=v_{k}t^{k}$
we notice that $v'_{k}tt$ satisfies the same conditions as $vtt$
thus there is a decomposition $sus$ such that $s=t$. Because $t$
irreducible then $l\left(v_{k}\right)\geq l\left(t\right)$. If $v_{k}=t$
we are done if not we write $v_{k}=tv_{k+1}$ and we examine the cyclic
permutation $v_{k+1}t^{k+1}tt$. We notice that $l\left(v_{k+1}\right)<l\left(v_{k}\right)$
thus this process must end at some point i.e. for some $k$ we have
$v_{k}=t$ thus $w$ has a cyclic permutation that is a power of $t$.
\end{proof}
\begin{prop}
\label{prop:Let--be}Let $\oton wm$ be cyclic words and let $\sum_{i=1}^{m}a_{i}w_{i}$
a linear combination that is not in $\left\langle lw-w^{l}|w\text{ cyclicaly reduced},\allowbreak l\in\mathbb{N}\right\rangle $.
Then there exists a $k$ s.t. $\pi_{k}\left(\sum_{i=1}^{m}a_{i}w_{i}\right)\neq0$.
This shows that 
\[
\ker\pi=\left\langle lw-w^{l}|w\text{ cyclicaly reduced},l\in\mathbb{N}\right\rangle 
\]
.
\end{prop}

\begin{proof}
Using the relations $lw-w^{l}$ the linear combination $\sum_{i=1}^{m}a_{i}w_{i}$
always has representative $\sum_{i=1}^{\hat{m}}\hat{a}_{i}\hat{w}_{i}$
such that $\oton w{\hat{m}}$ are pairwise distinct words that are
not powers. Therefore we assume this for $\sum_{l=1}^{m}a_{l}w_{l}$.
Without loss of generality assume $w_{m}$ is of maximal length. We
notice that the length $k$ subwords of a length $k$ cyclic word
are its cyclically reduced representative. If two length $k$ cyclic
words have a length $k$ subword in common they are the same cyclic
word, therefore we can assume that $l\left(w_{m}\right)>l\left(w_{i}\right)$
for $1\leq i\leq m-1$. We notice that if $u$ is a subword of a cyclic
word $w$ with $l\left(u\right)>l\left(w\right)$ then there is a
cyclically reduced representative $w'$ of $w$ such that $u=w'^{r}w_{1}$
where $w_{1}$ is a proper prefix of $w'$ and $r\in\mathbb{N}$ ($w_{1}$
is possibly trivial but not when $r=1$). Thus $u$ has the form $t\cdot v\cdot t$.
Let $l\left(w_{m}\right)=k$. Since $w_{m}$ is not a power by Lemma
\ref{lem:notPower} it has a subword of length $k$ that does not
have the form $t\cdot v\cdot t$. Assume that $\pi_{k}\left(\sum_{i=1}^{m}a_{i}w_{i}\right)=0$
then $-a_{m}\pi_{k}\left(w_{m}\right)=\sum_{i=1}^{m-1}a_{i}\pi_{k}\left(w_{i}\right)$
but this is a contradiction because the sum on the left side has a
u-word that does not have the form $tvt$ while all the u-word summands
on the right side have the form $tvt$ since $k=l\left(w_{m}\right)>l\left(w_{i}\right)$
for $1\leq i\leq m-1$. 
\end{proof}
This proposition is the main component of Theorem \ref{thm:main_injective_representation}.
We take a moment to talk about the topology of an inverse limit of
free modules.
\begin{defn}
Let $p_{i}:M_{i}\to M_{i-1}$ be an inverse system of finitely generated
free $\mathbb{Z}$-module (vector spaces). We define a valuation $V$
on $\varprojlim M_{i}$. Let $v=\left(v_{i}\right)_{i\in\mathbb{N}}\in\varprojlim M_{i}$
i.e. $v_{i}=p_{i+1}\left(v_{i+1}\right)$ than we define $V\left(v\right)=\max\left\{ k\in\mathbb{N}|v_{k}=0\right\} $.
Let $\hat{p}_{i}:\varprojlim M_{i}\to M_{i}$ the natural projection
and note that $\hat{p}_{i}=p_{i+1}\circ\hat{p}_{i+1}$. 
\end{defn}

Three equivalent ways to define the topology of $\varprojlim M_{i}$
\begin{enumerate}
\item We take the discrete topology on each $M_{i}$ and define the topology
of $\varprojlim M_{i}\leq\prod_{i}M_{i}$ to be the induced topology
from the product topology (general topology)
\item We take topology induced by taking $\ker\hat{p}_{i+1}\leq\ker\hat{p}_{i}$
to be local base of open neighborhoods of $0$. (topological module)
\item Let $v\in\varprojlim M_{i}$ and $n\in\mathbb{N}$ we define $B_{v,n}=\left\{ w\in\varprojlim M_{i}|V\left(w-v\right)\geq n\right\} $.
The sets $B_{v,n}$ are a base for the topology. (metric topology) 
\end{enumerate}
Let us return to the morphism $\pi:\mathbb{Z}\left[C\right]\to\varprojlim\tilde{M}_{i}$.
For every $i$ the morphism $\pi_{i}$ is onto $\tilde{M}_{i}$ thus
the image of $\pi$ is dense in $\varprojlim\tilde{M}_{i}$. The module\noun{
}$\mathbb{Z}\left[C\right]/\left\langle lw-w^{l}|w\text{ cyclicaly reduced},l\in\mathbb{N}\right\rangle $
can be seen as the free module generated by cyclic words that are
not non-trivial powers. The usual way to define a metric would be
$\norm{}:\varprojlim\tilde{M}_{i}\to\mathbb{R}$ by $\norm v=2^{-V\left(v\right)}$
 In this case $\varprojlim\tilde{M}_{i}$ is the completion of $\im\pi$
with respect to the metric. 

\subsection{Representation}

We now show how Theorem \ref{thm:main_counting_words} gives us a
representation. Let $\varphi$ be an automorphism. We recall the definition
of $m_{k}^{\varphi}$ from the introduction. Let $u\in W_{k}$, let
$S_{\varphi,u}$ be the multiset of words given by Theorem \ref{thm:main_counting_words},
let $w$ a cyclic word and $k_{v}=l\left(v\right)$ for $v\in S_{\varphi,u}$
then by definition of $\pi_{k}$ and Theorem \ref{thm:main_counting_words}
we have 
\[
u^{*}\left(\pi_{k}\left(\varphi\left(w\right)\right)\right)=\left|\hom\left(u,\varphi\left(w\right)\right)\right|=\sum_{v\in S_{\varphi,u}}\left|\hom\left(v,w\right)\right|=\sum_{v\in S_{\varphi,u}}v^{*}\left(\pi_{k_{v}}\left(w\right)\right)
\]
Because the equation is true for every $w$ we have an equation of
functions $u^{*}\circ\pi_{k}\circ\varphi=\sum v^{*}\circ\pi_{k_{v}}$
equivalently we can write 
\[
\varphi^{*}\circ\pi_{k}^{*}\left(u^{*}\right)=\sum_{v\in S_{\varphi,u}}\pi_{k_{v}}^{*}\left(v^{*}\right)
\]
Where we consider $\varphi$ as the linear extension of $\varphi$
to $\mathbb{Z}\left[C\right]$. We denote 
\[
p_{k,l}=p_{l+1}\circ\cdots\circ p_{k-1}\circ p_{k}
\]
 thus $p_{k,l}\circ\pi_{k}=\pi_{l}$. We recall 
\[
\overleftarrow{p}_{m_{k},k_{v}}\circ\pi_{m_{k}}=p_{m_{k},k_{v}}\circ\pi_{m_{k}}=\pi_{k_{v}}
\]
 so $\pi_{k_{v}}^{*}=\pi_{m_{k}}^{*}\circ\overleftarrow{p}_{m_{k},k_{v}}^{*}$
thus 
\[
\varphi^{*}\circ\pi_{k}^{*}\left(u^{*}\right)=\sum_{v\in S_{\varphi,u}}\pi_{k_{v}}^{*}\left(v^{*}\right)=\pi_{m_{k}}^{*}\left(\sum_{v\in S_{\varphi,u}}\overleftarrow{p}_{m_{k},k_{v}}^{*}\left(v^{*}\right)\right)
\]
 We define $\varphi_{k}^{*}:M_{k}^{*}\to M_{m_{k}}^{*}$ by 
\[
\varphi_{k}^{*}\left(u^{*}\right)=\sum_{v\in S_{\varphi,u}}\overleftarrow{p}_{m_{k},k_{v}}^{*}\left(v^{*}\right)
\]
 it satisfies $\varphi^{*}\circ\pi_{k}^{*}=\pi_{m_{k}}^{*}\circ\varphi_{k}^{*}$.
The homomorphism $\varphi_{k}^{*}$ defines a homomorphism $\varphi_{k}:M_{m_{k}}\to M_{k}$
that satisfies 
\[
\varphi_{k}\circ\pi_{m_{k}}=\pi_{k}\circ\varphi
\]
 We notice that we use $\overleftarrow{p}_{m_{k},k_{v}}^{*}$ in the
definition and we recall that $\overleftarrow{p}_{m_{k},k_{v}}^{*}$
is defined using an arbitrary u-word orientation. Thus a morphism
$\varphi_{k}$ that satisfies $\varphi_{k}\circ\pi_{m_{k}}=\pi_{k}\circ\varphi$
is not unique. A different orientation would result in a different
morphism $\varphi_{k}$. But if we consider the restriction $\varphi_{k}|_{\tilde{M}_{m_{k}}}:\tilde{M}_{m_{k}}\to\tilde{M}_{k}$
then it is unique. Let $\varphi_{k}$ and $\varphi'_{k}$ both satisfy
\[
\varphi_{k}\circ\pi_{m_{k}}=\pi_{k}\circ\varphi=\varphi'_{k}\circ\pi_{m_{k}}
\]
 but $\pi_{m_{k}}$is onto on $\tilde{M}_{m_{k}}$ so $\varphi'_{k}|_{\tilde{M}_{m_{k}}}=\varphi_{k}|_{\tilde{M}_{m_{k}}}$.
We notice that 
\begin{align*}
p_{k}\circ\varphi_{k}\circ\pi_{m_{k}} & =p_{k}\circ\pi_{k}\circ\varphi\\
 & =\pi_{k-1}\circ\varphi\\
 & =\varphi_{k-1}\circ\pi_{m_{k-1}}\\
 & =\varphi_{k-1}\circ p_{m_{k},m_{k-1}}\circ\pi_{m_{k}}
\end{align*}
But $\pi_{m_{k}}$ is onto on $\tilde{M}_{m_{k}}$ so 
\begin{equation}
p_{k}\circ\varphi_{k}=\varphi_{k-1}\circ p_{m_{k},m_{k-1}}\label{eq:presentation}
\end{equation}
 This shows that the sequence $\varphi_{k}:\tilde{M}_{m_{k}}\to\tilde{M}_{k}$
defines a homomorphism $\varprojlim\tilde{M}_{k}\to\varprojlim\tilde{M}_{k}$.
 We notice that 
\begin{align*}
\left(\varphi\circ\psi\right)_{k}\circ p_{m_{m_{k}^{\varphi}}^{\psi},m_{k}^{\varphi\circ\psi}}\circ\pi_{m_{m_{k}^{\varphi}}^{\psi}} & =\left(\varphi\circ\psi\right)_{k}\circ\pi_{m_{k}^{\varphi\circ\psi}}\\
 & =\pi_{k}\circ\left(\varphi\circ\psi\right)\\
 & =\varphi_{k}\circ\pi_{m_{k}^{\psi}}\circ\psi\\
 & =\varphi_{k}\circ\psi_{m_{k}^{\varphi}}\circ\pi_{m_{m_{k}^{\varphi}}^{\psi}}
\end{align*}
 as before we get 
\begin{equation}
\left(\varphi\circ\psi\right)_{k}\circ p_{m_{m_{k}^{\varphi}}^{\psi},m_{k}^{\varphi\circ\psi}}=\varphi_{k}\circ\psi_{m_{k}^{\varphi}}\label{eq:composition}
\end{equation}
 We notice that the sequences $\left(\varphi\circ\psi\right)_{k}\circ p_{m_{m_{k}^{\varphi}}^{\psi},m_{k}^{\varphi\circ\psi}}$
and $\left(\varphi\circ\psi\right)_{k}$ define the same homomorphism
in the limit. Thus we get that $\text{Out}\left(F_{n}\right)\to\aut\left(\varprojlim\tilde{M}_{k}\right)$
is a homomorphism. Lastly let $\varphi,\psi$ be automorphism from
different conjugacy classes then there exists a generator of the free
group $x$ such that $\varphi\left(x\right),\psi\left(x\right)$ are
from different conjugacy classes. because $\varphi,\psi$ are automorphisms
$\psi\left(x\right),\varphi\left(x\right)$ are not powers thus by
Proposition \ref{prop:Let--be} there is a $k$ such that $\pi_{k}\circ\varphi\left(x\right)\neq\pi_{k}\circ\psi\left(x\right)$
without loss of generality assume $m_{k}^{\varphi}\leq m_{k}^{\psi}$
thus 

\begin{eqnarray*}
\varphi_{k}\circ p_{m_{k}^{\psi},m_{k}^{\varphi}}\left(\pi_{m_{k}^{\psi}}\left(x\right)\right) & \neq & \psi_{k}\left(\pi_{m_{k}^{\psi}}\left(x\right)\right)
\end{eqnarray*}
 The sequences $\varphi_{k}\circ p_{m_{k}^{\psi},m_{k}^{\varphi}}$
and $\varphi_{k}$ define the same homomorphism in the limit. We conclude
that $\text{Out}\left(F_{n}\right)\to\aut\left(\varprojlim\tilde{M}_{k}\right)$
is injective (Theorem \ref{thm:main_injective_representation}). We
note that $M_{k}$ and $M_{m_{k}}$ have a given basis thus $\varphi_{k}:M_{m_{k}}\to M_{k}$
is a matrix and $\varphi$ is represented by a tower of finite matrices.
Equation \ref{eq:presentation} shows that $\varphi_{k-1}$ can be
reconstructed given $\varphi_{k}$ thus each successive matrix gives
more information about $\varphi$. Equation \ref{eq:composition}
shows that computations in $\aut\left(\varprojlim\tilde{M}_{k}\right)$
can be done via matrix multiplication. 

\bibliographystyle{plain}
\bibliography{mybib}

\end{document}